\documentclass{article}
\usepackage{amssymb,amsthm,proof,mathrsfs,textcomp,setspace,tikz,comment,mathtools,stackengine,enumitem,authblk}

\newtheorem{thm}{Theorem}  
\newtheorem{cor}[thm]{Corollary}
\newtheorem{lemma}[thm]{Lemma}
\theoremstyle{definition}
\newtheorem{defn}[thm]{Definition}

\def\L{\mathcal{L}}
\def\D{\mathcal{D}}

\def\g{\Gamma}
\def\d{\Delta}
\def\To{\Rightarrow}
\newcommand{\GSD}{\Gamma\Rightarrow\Delta}
\newcommand{\infrule}[1]{\scriptstyle\it{#1}}
\newcommand{\prt}{\; ; \;}

\title{Interpolation in extensions of first-order logic}
\author{G. Gherardi}
\affil{Dipartimento di Filosofia e Comunicazione, Universit\`{a} di Bologna, guido.gherardi@unibo.it}
\author{P. Maffezioli}
\affil{Institut f\"{u}r Philosophie II, Ruhr-Universit\"{a}t Bochum, paolo.maffezioli@ruhr-uni-bochum.de}
\author{E. Orlandelli}
\affil{Dipartimento di Filosofia e Comunicazione, Universit\`{a} di Bologna, eugenio.orlandelli@unibo.it}
\date{}

\begin{document}
 
 \maketitle
 
 \begin{abstract}
 We prove a generalization of Maehara's lemma to show that the extensions of classical and intuitionistic first-order logic with a special type of geometric axioms, called singular geometric axioms, have Craig's interpolation property. As a corollary, we obtain a direct proof of interpolation for (classical and intuitionistic) first-order logic with identity, as well as interpolation for several mathematical theories, including the theory of equivalence relations, (strict) partial and linear orders, and various intuitionistic order theories such as apartness and positive partial and linear orders. 
 \end{abstract}

Craig's interpolation theorem \cite{Craig1957} is a central result in first-order logic. It asserts that for any theorem $A \to B$ there exists a formula $C$, called \emph{interpolant}, such that $A \to C$ and $C \to B$ are also theorems and $C$ only contains non-logical symbols that are contained in both $A$ and $B$ (and if $A$ and $B$ have no non-logical symbols in common, then either $\neg A$ is a theorem or $B$ is). The aim of this paper is to extend interpolation beyond first-order logic. In particular, we show how to prove interpolation in extensions of intuitionistic and classical sequent calculi with \emph{singular geometric rules}, a special case of geometric rules investigated in \cite{Negri2003}. Interpolation for singular geometric rules will be obtained by generalizing a standard result, reportedly due to Maehara in \cite{Takeuti1987} and known as \textquotedblleft Maehara's lemma\textquotedblright{} \cite{Maehara1960}.\footnote{In this work we shall not consider semantic methods to prove interpolation. These have been applied extensively to non-classical logics in \cite{GabbayMaksimova2005}; there are also proofs of interpolation for non-classical logics that are more similar to our approach, especially \cite{BaazIemhoff2005,FittingKuznets2015,Kuznets2018}.} 

The proof of Maehara's lemma for intuitionistic and classical first-order logic requires cut elimination. This clearly challenges the project of proving the lemma for systems extending first-order logic with  axioms, since such systems are not generally cut-free (cf.~\cite[\S 4.5]{TroelstraSchwichtenberg2000} and \cite [\S 6.3]{NegrivonPlato2001} for  different approaches to non-logical axioms).  
For example, in the calculus $\mathsf{LK_e}$, an extension of Gentzen's $\mathsf{LK}$ for first-order logic with identity, cuts on identities $s = t$ are not eliminable (cf. Theorem 6 in \cite{Takeuti1987}, where these cuts are called \textquotedblleft inessential\textquotedblright{}). Fortunately, interpolation can still be proved for first-order logic with identity. The drawback of the existing proofs, however, is that
they are indirect, in the sense that the interpolant is not built using exclusively the rules of the calculus. In \cite{TroelstraSchwichtenberg2000}, for example, a translation is used to reduce interpolation for first-order logic with identity to the case of pure first-order logic.\footnote{For other proofs of interpolation via translation see \cite{RasgaCarnielliSernadas2009} and \cite{BonacinaJohansson2015}.} A different route is taken in \cite{Gallier2015}, using the method of \textquotedblleft axioms in the context,\textquotedblright{} where interpolation is again not proved directly in $\mathsf{LK_e}$, but in an variant of $\mathsf{LK}$, equivalent to $\mathsf{LK_e}$, in which all derivable sequents have the axioms governing the identity predicate in the context.\footnote{Thanks to a referee for bringing this to our attention.} Beside the use of such indirect maneuvers, these approaches are specifically designed for first-order logic with identity and it is not entirely obvious how to adapt them to other extensions of first-order logic. On the other hand, in this paper interpolation is proved via a generalization of Maehara's lemma to a class of extensions of first-order logic (which include first-order logic with identity as a particular case) and using no other means than the rules of the calculus (Lemma \ref{th:maehara_singular}).  
 
Our generalization of Maehara's lemma is based on previous work by Negri and von Plato who have shown (in a series of papers starting from \cite{NegrivonPlato1998}) how to recover cut elimination (as well as the admissibility of other structural rules) for extensions of the calculi $\mathsf{G3c}$ and $\mathsf{m\textnormal{-}G3i}$ for classical and intuitionistic first-order logic. Of particular interest for the present work are the extensions with geometric rules, investigated in \cite{Negri2003}.\footnote{We depart from Negri's approach in taking the intuitionistic single-succedent calculus $\mathsf{G3i}$ instead of the multi-succedent $\mathsf{m\textnormal{-}G3i}$ of \cite{Negri2003}; in Theorem \ref{prop:structural_properties_geometric_in} we will also prove, along the way, that cut elimination holds for geometric extensions of $\mathsf{G3i}$.} Once cut elimination is recovered in this way, we impose a singularity condition on geometric rules to isolate those containing at most one non-logical predicate (identity will be counted as logical). Our main result is to show that Maehara's lemma holds when $\mathsf{G3c}$ and $\mathsf{G3i}$ are extended with singular geometric rules (Lemma \ref{th:maehara_singular}). Then interpolation follows easily from the generalized Maehara's lemma (Theorem \ref{CraigSingular}). Finally, we consider applications of Theorem \ref{CraigSingular} and we show that singular geometric rules include many interesting extensions of intuitionistic and classical first-order logic, especially (classical and intuitionistic) first-order logic with identity, the theory of equivalence relations, (strict) partial and linear orders, 
the theory of apartness and the theory of positive partial and linear orders. 

 \section{Classical and intuitionistic sequent calculi}
 
 The language $\L$ is a first-order language with individual constants and no functional symbols. Terms ($s,t,u,\dots$) are either variables ($x,y,z,\dots$) or individual constants ($a,b,c\dots$). $\L$ contains also denumerably many $k$-ary predicates $P^k,Q^k,R^k, \dots$ for each $k \geq 0$. $\L$ may also contain the identity. We agree that all predicates, except identity, are non-logical. Moreover, it is convenient to have two propositional constants $\bot$ (falsity) and $\top$ (truth). Formulas are built up from atoms $P^k (t_1, \dots , t_k)$, the constants $\bot$ and $\top$ using logical operators $\wedge$, $\vee$, $\to$, $\exists$ and $\forall$ as usual. We use $P,Q,R,\dots$ for atoms, $A,B,C,\dots$ for formulas and $\Gamma, \Delta, \Pi , \dots$ for (possibly empty) finite multisets of formulas. The negation $\neg A$ of a formula $A$ is defined as $A \to \bot$. We also agree that $\Gamma,\Delta$ is an abbreviation for $\Gamma \cup \Delta$ (where $\cup$ is the multiset union) and $\bigwedge\Gamma$ ($\bigvee\Gamma$) stands for the conjunction (disjunction, respectively) of all formulas in $\Gamma$. Moreover, if $\Gamma$ is empty, then $\bigwedge\Gamma\equiv\top$ and $\bigvee\Gamma\equiv\bot$, where $\equiv$ indicates syntactic identity (up to $\alpha$-congruence) between expressions of the object-language. 
 
 The substitution of a variable $x$ with a term $t$ in a term $s$ (in a formula $A$, in a multiset $\Gamma$) will be indicated as $s [\,{}_{x}^{t}\,]$ ($A [\,{}_{x}^{t}\,]$ and $\Gamma [\,{}_{x}^{t}\,]$, respectively) and defined as usual. To indicate the simultaneous substitution of the list of variables $x_1 , \dots , x_n$ (abbreviated in $\bar{x}$) with the list of terms $t_1, \dots , t_n$ (abbreviated in $\bar{t}$), we use $[\,{}_{\bar{x}}^{\bar{t}}\,]$ in place of $[\,{}_{x_1\,\dots\,x_n}^{t_1\,\dots\,t_n}\,]$. Later on, we shall also need a more general notion of substitution of terms for terms (not just variables) which will be proved to preserve derivability (Lemma \ref{th:substitution_general}). 
 
Finally, let $\mathsf{FV} (A)$ be the set of free variables of a formula $A$ and let $\mathsf{Con} (A)$ be the set of its individual constants. We agree that the set of terms $\mathsf{Ter} (A)$ of $A$ is $\mathsf{FV} (A) \cup \mathsf{Con} (A)$. Moreover, if $\mathsf{Rel}(A)$ is the set of non-logical predicates of $A$ then we define the language $\L(A)$ of $A$ as $\mathsf{Ter} (A) \cup \mathsf{Rel}(A)$. Notice that $= \; \notin \mathcal{L} (A)$, for all $A$. Such notions are immediately extended to multisets of formulas $\Gamma$, by letting $\mathsf{FV} (\Gamma)$ to be defined as $\bigcup_{A\in\Gamma} \mathsf{FV} (A)$, and analogously for $\mathsf{Con} (\Gamma)$, $\mathsf{Ter} (\Gamma)$, $\mathsf{Rel}(\Gamma)$ and $\L (\Gamma)$.
  
 The calculus $\mathsf{Gc}$ ($\mathsf{Gi}$) is a variant of $\mathsf{LK}$ ($\mathsf{LI}$) for classical (intuitionistic, respectively) first-order logic, originally introduced by Gentzen in \cite{Gentzen1969a}. In the literature, especially in \cite{TroelstraSchwichtenberg2000} and \cite{NegrivonPlato2001}, $\mathsf{Gc}$ and $\mathsf{Gi}$ are commonly referred to as $\mathsf{G3c}$ and $\mathsf{G3i}$ but we will omit \textquoteleft $\mathsf{3}$\textquoteright{} in the interest of readability. Moreover, we will write $\mathsf{G}$ to refer to either $\mathsf{Gc}$ or $\mathsf{Gi}$. A sequent in $\mathsf{Gc}$ is a pair $\langle \Gamma,\Delta \rangle$ of multisets, indicated as $\GSD$. The calculus $\mathsf{Gc}$ consists of the following initial sequents and logical rules (where $y$ is an \emph{eigenvariable} in $R\forall$ and $L\exists$, i.e. $y$ must not occur free in the conclusion of these rules):

 \begin{center}
\begin{tabular}{cc}
\multicolumn{2}{c}{\emph{The calculus $\mathsf{Gc}$}}\\\\
\multicolumn{2}{c}{$P , \GSD , P$}\\\\
$\infer[\infrule{L\bot}]{\bot , \GSD}{}$ & $\infer[\infrule{R\top}]{\GSD , \top}{}$  \\\\

%
%

$\infer[\infrule{L\wedge}]{A \wedge B , \GSD}{A , B , \GSD}$

&

$\infer[\infrule{R\wedge}]{\GSD , A \wedge B}{\GSD , A & \GSD , B}$

\\\\

$\infer[\infrule{L\vee}]{A \vee B , \GSD}{A , \GSD & B , \GSD}$

&

$\infer[\infrule{R\vee}]{\GSD , A \vee B}{\GSD , A , B}$

\\\\

$\infer[\infrule{L\to}]{A \to B , \GSD}{\GSD , A & B , \GSD}$

&

$\infer[\infrule{R\to}]{\GSD , A \to B}{A , \GSD , B}$

\\\\

$\infer[\infrule{L\forall}]{\forall x A , \GSD}{A [\, {}_{x}^{t} \,] , \forall x A , \GSD}$

&

$\infer[\infrule{R\forall}]{\GSD ,\forall x A}{\GSD , A [\,{}_{x}^{y}\,]}$

\\\\

$\infer[\infrule{L\exists}]{\exists x A , \GSD}{A [\,{}_{x}^{y}\,] , \GSD}$

&

$\infer[\infrule{R\exists}]{\GSD ,\exists x A}{\GSD , \exists x A , A [\,{}_{x}^{t}\,]}$

\end{tabular}
\end{center}

Sequents in $\mathsf{Gi}$ are defined as in $\mathsf{Gc}$, except that $\Delta$ must contain exactly one formula. The calculus $\mathsf{Gi}$ has the following initial sequents and logical rules (again, $y$ is an \emph{eigenvariable} in $R\forall$ and $L\exists$). 

 \begin{center}
\begin{tabular}{cc}
\multicolumn{2}{c}{\emph{The calculus $\mathsf{Gi}$}}\\\\
\multicolumn{2}{c}{$P , \Gamma \Rightarrow P$}\\\\
$\infer[\infrule{L\bot}]{\bot , \Gamma \Rightarrow C}{}$&$\infer[\infrule{R\top}]{\Gamma \Rightarrow \top}{}$  \\\\


$\infer[\infrule{L\wedge}]{A \wedge B , \Gamma \Rightarrow C}{A , B , \Gamma \Rightarrow C}$

&

$\infer[\infrule{R\wedge}]{\Gamma \Rightarrow A \wedge B}{\Gamma \Rightarrow A & \Gamma \Rightarrow B}$

\\\\

$\infer[\infrule{L\vee}]{A \vee B , \Gamma \Rightarrow C}{A , \Gamma \Rightarrow C & B , \Gamma \Rightarrow C}$

&

$\infer[\infrule{R\vee_1}]{\Gamma \Rightarrow A \vee B}{\Gamma \Rightarrow A}$ \quad $\infer[\infrule{R\vee_2}]{\Gamma \Rightarrow A \vee B}{\Gamma \Rightarrow B}$

\\\\

$\infer[\infrule{L\to}]{A \to B , \Gamma \Rightarrow C}{A \to B , \Gamma \Rightarrow A & B , \Gamma \Rightarrow C}$

&

$\infer[\infrule{R\to}]{\Gamma \Rightarrow A \to B}{A , \Gamma \Rightarrow B}$

\\\\

$\infer[\infrule{L\forall}]{\forall x A , \Gamma \Rightarrow C}{A [\, {}_{x}^{t} \,] , \forall x A , \Gamma \Rightarrow C}$

&

$\infer[\infrule{R\forall}]{\Gamma \Rightarrow \forall x A}{\Gamma \Rightarrow A [\,{}_{x}^{y}\,]}$

\\\\

$\infer[\infrule{L\exists}]{\exists x A , \Gamma \Rightarrow C}{A [\,{}_{x}^{y}\,] , \Gamma \Rightarrow C}$

&

$\infer[\infrule{R\exists}]{\Gamma \Rightarrow \exists x A}{\Gamma \Rightarrow A [\,{}_{x}^{t}\,]}$

\end{tabular}
\end{center}

A derivation in $\mathsf{G}$ is a tree of sequents which grows according to the rules of $\mathsf{G}$ and whose leaves are initial sequents or conclusions of a 0-premise rule. A derivation of a sequent is a derivation concluding that sequent and a sequent is derivable when there is a derivation of it.  As usual, we consider only \emph{pure-variable derivations}: bound and free variables are kept distinct, and no two rule instances have the same variable as \emph{eigenvariable}, see \cite[p. 38]{TroelstraSchwichtenberg2000}. 
The height $h$ of a derivation is defined inductively as follows: the derivation height of an initial sequent or of a conclusion of a 0-premise rule is 0, the derivation height of a derivation of a conclusion of a one-premise rule is the derivation height of its premise plus 1, and the derivation height of a derivation of a conclusion of a $n$-premise rule ($n \geq 2$) is the maximum of the derivation heights of its premises plus 1. A sequent is $h$-derivable if it is derivable with a derivation of height less than or equal to $h$. A rule is admissible if the conclusion is derivable whenever the premises are derivable; a rule is height-preserving admissible if the conclusion is $h$-derivable whenever the premises are $h$-derivable. 
Derivations will be denoted by $\D , \D_1, \D_2, \dots$. We agree to use $\D \vdash \GSD$ to indicate that $\D$ is a derivation in $\mathsf{G}$ of $\GSD$ and $\vdash \GSD$ to indicate that $\GSD$ is derivable; finally, $\vdash^h \GSD$ indicates that $\GSD$ is $h$-derivable. We will use a double-line rule of the form
$$
\infer=[\infrule R]{\Gamma\To\Delta}{\Pi\To\Sigma}
$$
to indicate that $\GSD$ is derivable from $\Pi\To\Sigma$ by  a (possibly empty) sequence of instances of the rule \mbox{\it{R}}.  It is easy to see that initial sequents with $A , \GSD , A$, for an arbitrary $A$, are derivable in $\mathsf{G}$ (where $\d$ is empty for $\mathsf{Gi}$). 

The following structural rules for $\mathsf{Gc}$ (weakening, contraction and cut) are valid in the standard semantics of $\mathsf{Gc}$. 

 \begin{center}
  \begin{tabular}{cc}
  
  \multicolumn{2}{c}{\emph{Structural rules of} $\mathsf{Gc}$}\\\\
  $\infer[\infrule{Wkn}]{A, \Gamma \Rightarrow \Delta}{\Gamma \Rightarrow \Delta}$ & $\infer[\infrule{Wkn}]{\Gamma \Rightarrow \Delta , A}{\Gamma \Rightarrow \Delta}$\\\\
  $\infer[\infrule{Ctr}]{A, \Gamma \Rightarrow \Delta}{A, A , \Gamma \Rightarrow \Delta}$ & $\infer[\infrule{Ctr}]{\Gamma \Rightarrow \Delta , A}{\Gamma \Rightarrow \Delta , A, A}$\\\\
  \multicolumn{2}{c}{$\infer[\infrule{Cut}]{\Gamma , \Pi \Rightarrow \Delta , \Sigma}{\Gamma \Rightarrow \Delta , A & A , \Pi \Rightarrow \Sigma}$}
  \end{tabular}

 \end{center}
 However, we can safely leave them out without impairing the completeness of $\mathsf{Gc}$, since they are all admissible in it. In fact, weakening and contraction are also height-preserving admissible. Regarding $\mathsf{Gi}$, the structural rules are:
 
  \begin{center}
  \begin{tabular}{cc}
  \multicolumn{2}{c}{\emph{Structural rules of} $\mathsf{Gi}$}\\\\
  $\infer[\infrule{Wkn}]{A, \Gamma \Rightarrow C}{\Gamma \Rightarrow C}$ & $\infer[\infrule{Ctr}]{A, \Gamma \Rightarrow C}{A, A , \Gamma \Rightarrow C}$\\\\
   \multicolumn{2}{c}{$\infer[\infrule{Cut}]{\Gamma , \Delta \Rightarrow C}{\Gamma \Rightarrow A & A , \Delta \Rightarrow C}$}
  \end{tabular}
 \end{center}
 These rules are also valid in the model-theoretic semantics for intuitionistic logic, but just like in the classical case, they are all admissible in $\mathsf{Gi}$ (again, weakening and contracting are height-preserving admissible) and there is no need to take any of them as primitive. The proof of the admissibility of the structural rules in any of the two calculi requires some preparatory results. First, the height-preserving admissibility of substitution in $\mathsf{G}$.  
 
 \begin{lemma}\label{sub}
   In $\mathsf{G}$, if $\, \vdash^h \GSD$ and $t$ is free for $x$ in $\Gamma , \Delta$ then  $\, \vdash^h \Gamma [\,{}_{x}^{t}\,] \Rightarrow \Delta [\,{}_{x}^{t}\,]$.
  \end{lemma}
  
  Second, the so-called inversion lemma. Intuitively, a rule is invertible when it can be applied backwards, from the conclusion to its premises, and it is height-preserving invertible when it is invertible with the preservation of the derivation height (for a precise definition of height-preserving invertible rule see \cite[p. 76-77]{TroelstraSchwichtenberg2000}). 
  
  \begin{lemma}\label{inv}
  In $\mathsf{Gc}$ all  rules are height-preserving invertible. In $\mathsf{Gi}$ all rules, except $R\vee$, $L\to$ and $R\exists$, are height-preserving invertible. However, $L\to$ is height-preserving invertible with respect to its right premise. 
  \end{lemma}
   
   With height-preserving admissibility of substitution and inversion lemma, it is possible to prove the admissibility of the structural rules. 
   
   \begin{thm}\label{prop:structural_properties_logic}
 In $\mathsf{G}$ weakening and contraction are height-preserving admissible. Moreover, cut is admissible. 
  
 \end{thm}
 
The proof of Lemma \ref{sub}, Lemma \ref{inv}, and  Theorem \ref{prop:structural_properties_logic} are standard and the interested reader is referred to \cite{TroelstraSchwichtenberg2000} and \cite{NegrivonPlato2001}.

 \subsection{From axioms to rules}
 
 Extensions of $\mathsf{G}$ are not, in general, cut free; this means that Theorem \ref{prop:structural_properties_logic} does not necessarily hold in the presence of new initial sequents or rules. For example, a natural way to extend $\mathsf{Gc}$ to cover first-order logic with identity is to allow derivations to start with initial sequents of the form $\Rightarrow s = s$ and $s = t , P [{}_{x}^{s}] \Rightarrow P [{}_{x}^{t}]$, corresponding to the reflexivity of identity and Leibniz's principle of indescernibility of identicals, respectively (we call these sequents $S_1$ and $S_2$). Notice that $S_2$ is in fact a scheme which becomes $s = t , s = s \Rightarrow  t = s$, when $P$ is $x = s$. From this, via cut on $\Rightarrow s = s$, one derives $s = t \Rightarrow  t = s$, namely the symmetry of identity. However,  such a sequent has no derivation without cut. Therefore, cut is not admissible in $\mathsf{Gc} + \{ S_1 , S_2 \}$, though it is admissible in the underlying system $\mathsf{Gc}$. 
 
In \cite{NegrivonPlato1998} Negri and von Plato have shown how to recover cut elimination for (classical) first-order logic with identity by transforming $S_1$ and $S_2$ into an equivalent pair of rules of the form: 
 
$$
\infer[\infrule{\textnormal{\emph{Ref}}}_=]{\Gamma\Rightarrow\Delta}{s=s,\Gamma\Rightarrow\Delta} \qquad \infer[\infrule{\textnormal{\emph{Repl}}}_=]{s=t,P[{}_{x}^{s}],\Gamma\Rightarrow\Delta}{P[{}_{x}^{t}],s=t,P[{}_{x}^{s}],\Gamma\Rightarrow\Delta}
$$

If one replaces $S_1$ and $S_2$ with the corresponding rules, it is easy to derive $s = t \Rightarrow t = s$ without any application of cut. More generally, cut elimination holds in $\mathsf{Gc} + \{\textnormal{\emph{Ref}} , \textnormal{\emph{Repl}} \}$ (cf. Theorem 4.2 in \cite{NegrivonPlato1998} and \cite[\S 6.5]{NegrivonPlato2001}).  This result can be, and has been, extended in different directions. Here we are particularly interested in the fact, established by \cite{Negri2003}, that cut elimination holds in extensions of $\mathsf{Gc}$ with geometric rules (of which the rules of identity are special cases). The result will be reviewed briefly below, while for a more thorough discussion on this topic the reader is referred to \cite{Negri2003} or the monograph \cite{NegrivonPlato2011}. 

In \cite{Negri2003} Negri also showed that cut elimination holds for geometric theories formulated as extensions of the multi-succedent calculus $\mathsf{m\textnormal{-}G3i}$ for intuitionistic logic, introduced in \cite{Dragalin1988}. For our purposes, however, it is better to work in $\mathsf{Gi}$ as the underlying logical calculus for intuitionistic logic. In this way we can rely on the proof of Maehara's lemma for $\mathsf{Gi}$ already available in the literature (whereas to our knowledge no attempt has been made to obtain a similar result for $\mathsf{m\textnormal{-}G3i}$). In fact, it is not entirely obvious how to prove Maehara's lemma for $\mathsf{m\textnormal{-}G3i}$. Working in $\mathsf{Gi}$ is thus more advantageous as far as Maehara's lemma is concerned, but one needs first to make sure that cut elimination holds in the presence of geometric rules. Thus, after introducing geometric rules, we will show that the standard cut-elimination procedures goes through with minor adjustment in geometric extensions of $\mathsf{Gi}$ (Theorem \ref{prop:structural_properties_geometric_in}).

  \subsection{Geometric theories}

  A \emph{geometric axiom} is a formula  following the \emph{geometric axiom scheme} below:

 $$\forall \bar{x} (P_1  \wedge \dots \wedge P_n \to \exists \bar{y}_1 M_1 \vee \dots \vee \exists \bar{y}_m  M_m)$$
 where each $P_j$ is an atom and each $M_i$ is a conjunction of a list of atoms $Q_{i_{1}} , \dots , Q_{i_{\ell}}$ and none of the variables in any $\bar{y}_i$ are free in the $P_j$s. We shall conveniently abbreviate $Q_{i_{1}} , \dots , Q_{i_{\ell}}$ in $\mathbf{Q}_i$. In a geometric axiom, if $m = 0$ then the consequent of $\to$ becomes $\bot$, whereas if $n = 0$ the antecedent of $\to$ becomes $\top$. A \emph{geometric theory} is a theory containing only geometric axioms. 
An $m$-premise \emph{geometric rule}, for $m\geq 0$, is a rule following the \emph{geometric rule scheme} below:
 $$
  \infer[\infrule{R}]{P_1 , \dots , P_n , \Gamma \Rightarrow \Delta}{\mathbf{Q}^*_1 , P_1 , \dots , P_n , \Gamma \Rightarrow \Delta & \cdots & \mathbf{Q}^*_m , P_1 , \dots , P_n , \Gamma \Rightarrow \Delta}
  $$
where each $\mathbf{Q}^*_i$ is obtained from $\mathbf{Q}_i$ by replacing every variable in $\bar {y}_i$ with a variable which does not occur free in the conclusion. Such variables will be called the \emph{eigenvariables} of $R$.  Without loss of generality, we assume that each $\bar{y}_i$ consists of a single variable. 
In sequent calculus a geometric theory can be formulated by adding on top of $\mathsf{G}$ finitely many geometric rules (recall that $\Delta$ contains exactly one formula in $\mathsf{Gi}$).

Moreover, geometric rules are assumed to satisfy a natural closure property for contraction (see \cite[6.1.7]{NegrivonPlato2001}).

\begin{defn}[Closure condition]\label{closurecondition} If a geometric extension $\mathsf{G}'$ of $\mathsf{G}$ contains a rule where a substitution instance of the principal formulas produces a rule with repetition of the form:  
  
  $$
   \infer[\infrule R]{P_1,\dots,P_{n-2},P,P,\GSD}{\mathbf{Q}_1^*,P_1,\dots,P_{n-2},P,P,\GSD & \cdots & \mathbf{Q}_m^*,P_1,\dots,P_{n-2},P,P,\GSD}
   $$
  then $\mathsf{G}'$ contains or is closed under the following contracted instance of the rule:
  
  $$
   \infer[\infrule R^c]{P_1,\dots,P_{n-2},P,\GSD}{\mathbf{Q}_1^*,P_1,\dots,P_{n-2},P,\GSD & \cdots & \mathbf{Q}_m^*,P_1,\dots,P_{n-2},P,\GSD}
   $$
   \end{defn}
   As an illustration, we consider the rule \emph{Trans}$_\leqslant$ in the theory $\mathsf{PO}$ (see $\S$ \ref{PO}):

   $$
   \infer[\infrule{Trans_\leqslant}]{s \leqslant t , t \leqslant u , \Gamma \xRightarrow{} \Delta}{s \leqslant u , s \leqslant t , t \leqslant u , \Gamma \xRightarrow{} \Delta}
   $$
   Clearly, as an instance of such a rule we have:  
   
   $$
   \infer[\infrule{Trans_\leqslant}]{s \leqslant s , s \leqslant s , \Gamma \xRightarrow{} \Delta}{s \leqslant s , s \leqslant s , s \leqslant s , \Gamma \xRightarrow{} \Delta}
   $$
   Hence  $\mathsf{PO}$ has to be closed under the following contracted instance
   
   $$
   \infer[\infrule{Trans_\leqslant^c}]{s \leqslant s ,  \Gamma \xRightarrow{} \Delta}{s \leqslant s ,  s \leqslant s , \Gamma \xRightarrow{} \Delta}
   $$
   For $\mathsf{PO}$  we don't need to add the contracted rule \emph{Trans}$_\leqslant^c$,  because it is admissible thanks to rule \emph{Ref}$_\leqslant$. In general, however, this is not the case.

Let $\mathsf{G^g}$ be any extension of $\mathsf{G}$ with finitely many geometric rules satisfying the closure condition (from now on, we will tacitly assume that the closure condition is always met). We now show that cut elimination and the admissibility of the structural rules hold in $\mathsf{G^g}$. Although we will heavily rely on \cite{Negri2003}, we start by introducing a more general notion of substitution that allows an arbitrary term $u$ (possibly a constant) to be replaced by a term $t$. In the presence of such general substitutions, special care is needed in order to maintain the height-preserving admissibility of substitutions. In particular, general substitutions are height-preserving admissible, provided that the replaced term $u$ does not occur essentially in the calculus. Intuitively, a term $u$ occurs essentially in a rule $R$ when $u$ cannot be replaced (by an arbitrary term), namely when $u$ is a constant and $u$ already occurs in the axiom from which $R$ is obtained. More precisely, 

\begin{defn}
A constant $u$ occurs essentially in a geometric axiom $A$ if and only if, for some $t\not\equiv u$, $A [\,{}_{u}^{t}\,]$ is not an instance of the axiom $A$.  
\end{defn}

\noindent We also agree that a term $u$ occurs essentially in a geometric rule $R$ when it does so in the corresponding axiom. For example, in the geometric axiom $\neg 1 \leqslant 0$ of non-degenerate partial orders (see \cite[p. 116]{NegrivonPlato2011}) both $1$ and $0$ occur essentially; hence they also occur essentially in the corresponding geometric rule \emph{Non-deg}:
$$
\infer[\infrule{\textnormal{\emph{Non-deg}}}]{1 \leqslant 0 , \GSD}{}
$$

%
%
%
%

%
Now we show that the general substitution $[\,{}_{u}^{t}\,]$ is height-preserving admissible in $\mathsf{G^g}$, provided that $u$ occurs essentially in none of its geometric rule.
 \begin{lemma}\label{th:substitution_general}
 In $\mathsf{G^g}$, if $\, \vdash^n \GSD$, $t$ is free for $u$ in $\Gamma,\Delta$, and $u$ does not occur essentially in any rule of $\mathsf{G^g}$, then $\,  \vdash^n \Gamma [{}_{u}^{t}] \Rightarrow \Delta [{}_{u}^{t}]$.
 \end{lemma}

 \begin{proof}
 If $u$ is a variable,  the claim holds by extending Lemma \ref{sub} to $\mathsf{G^g}$. Otherwise, let $u$ be an individual constant. We can think of the derivation $\D$ of $\g\To\d$ as
 $$
 \infer[{[{}^{u}_{z}]}]{\g'[{}_{z}^{u}]\To\d'[{}^{u}_{z}]}{\g'\To\d'}
 $$
 where $\g'\To\d'$ is like $\g\To\d$ save that it has a fresh variable $z$ in place of $u$. Note that this is always feasible for purely logical derivations, and it is feasible for derivations involving geometric rules as long as these rules do not involve essentially the constant $u$. We transform $\D$ into
  $$
 \infer[{[{}^{t}_{z}]}]{\g'[{}_{z}^{t}]\To\d'[{}^{t}_{z}]}{\g'\To\d'}
 $$
 where $t$ is free for $z$ since we assumed it is free for $u$ in $\GSD$. We have thus found a derivation ($\D[{}_{u}^{t}]$) of $\Gamma [{}_{u}^{t}] \Rightarrow \Delta [{}_{u}^{t}]$ that has the same height as the derivation $\D$ of $\g\To\d$.
 \end{proof}
 
We can now show that Lemma \ref{inv} and Theorem \ref{prop:structural_properties_logic} still hold in $\mathsf{G^g}$. In fact, for $\mathsf{Gc^g}$ a proof has already been given in \cite{Negri2003}. 
 
 \begin{thm}[Negri]\label{prop:structural_properties_geometric_cl}
In $\mathsf{Gc^g}$ all the geometric and logical rules are height-preserving invertible. Moreover, weakening and contraction are height-preserving admissible and cut is admissible. 
  \end{thm}

 At this point we need to show that the same holds for $\mathsf{Gi}$. A similar result has been proved by Negri in \cite{Negri1999} for a subclass of geometric rules, called universal rules. In fact, Negri only considers specific instances of universal rules expressing  the axioms of the constructive theory of apartness and excess, see \S \ref{apart} and \S \ref{sectPPO}. Moreover, in \cite{Negri1999} only the quantifier-free version of $\mathsf{Gi}$ is considered. Here we extend Negri's result and show the admissibility of the structural rules for the full calculus $\mathsf{Gi}$ extended by arbitrary geometric rules. Then,  
 
  \begin{thm}\label{prop:structural_properties_geometric_in}
  In $\mathsf{Gi^g}$ all the geometric rules and all  logical rules, except $R\vee$, $L\to$ and $R\exists$, are height-preserving invertible. However, $L\to$ is height-preserving invertible with respect to its right premise. Moreover, weakening and contraction are height-preserving admissible and cut is admissible. 
 \end{thm}
 
 \begin{proof}
 
 The proof of height-preserving invertibility of the geometric and logical rules for $\mathsf{Gi^g}$ does not differ substantially from that for $\mathsf{Gi}$ and is left to the reader. We take a closer look at the admissibility of the structural rules. 
  
   
  \

  \emph{Weakening}.  To show that weakening is height-preserving admissible in $\mathsf{Gi^g}$, we need to extend the proof for $\mathsf{Gi}$ with the cases arising from geometric rules $R$. These cases be dealt with as for geometric rules over $\mathsf{m\textnormal{-}Gi}$ and $\mathsf{Gc}$ \cite[Thm. 2]{Negri2003}. In particular, if $R$ is an $m$-premises ($m\geq 1$) geometric rule with a variable condition on $y$, we replace $y$ with a fresh variable not occurring in the weakening formula, then we apply the inductive hypothesis and, finally, we apply $R$. If $R$ is an $m$-premises ($m\geq 1$) geometric rule without variable condition,  we can apply directly the inductive hypothesis and then $R$. Finally, if $R$ is a 0-premise geometric rule, the conclusion of weakening is obtained directly by $R$.

 \
 
 \emph{Contraction}. Once again, the new cases arising by the addition of geometric rules to $\mathsf{Gi}$ are similar to the cases in which these rules are added to $\mathsf{m\textnormal{-}Gi}$ or to $\mathsf{Gc}$ \cite[Thm. 4]{Negri2003}. This means we have three cases: of the occurrences of the contraction formula either (i) none, or (ii) exactly one, or (iii) both are principal in the final step of the derivation of the premise. The first case can be dealt with by induction, the second by inversion, and the third by the closure condition.   

 \
 
\emph{Cut}. To show that cut is admissible we need to prove that if $\, \vdash \Gamma \Rightarrow A$ and $\, \vdash A, \Delta \Rightarrow C$ then  $\, \vdash  \Gamma , \Delta \Rightarrow C$. The proof is by induction on the weight of the cut formula $A$ with a sub-induction on the sum of  heights of  derivation of the two premises (cut-height, for short). As for the proof of the admissibility of Cut over $\mathsf{m\textnormal{-}Gi^g}$ \cite[Thm. 5]{Negri2003}, we consider only the new cases arising from the geometric rules $R$.
\begin{enumerate}
\item The left premise of \emph{Cut} is by a 0-premise geometric rule $R$. Hence also the conclusion of \emph{Cut} is a conclusion of an instance of $R$.
\item The right premise is by a 0-premise geometric rule $R$ and the cut formula is not principal in it. We proceed as in case 1. 
\item The right premise is by an instance of a 0-premise geometric rule $R$ and the cut formula is  principal in it. In this case we know that $A$ is atomic (or $\top$ or $\bot$) and we consider the last step of the derivation of the left premise. If it is by a 0-premise (logic or geometric) rule or  it is an initial sequent, we proceed as in case 1.
\footnote{Observe that, unlike the cases of $\mathsf{m\textnormal{-}Gi^g}$ and $\mathsf{Gc^g}$, the cut formula $A$ must be principal in the left premise when this premise is an initial sequent.} If the left premise is by an $m$-premises ($m\geq 1$) logical or geometric rule, then  the cut formula is not principal in it  and we can permute the \emph{cut} upwards in the left premise (if the last rule applied in the left premise has \emph{eigenvariables}, we rename them before permuting the cut to avoid clashes). 
\item If the cut formula is not principal either in the left or in the right premise and this premise is by an $m$-premises  (for $m\geq 1$) geometric rule $R$, then, after having renamed any \emph{eigenvariable} of $R$ to avoid clashes, we permute the cut upwards with respect to this premise.
\item Finally, if the cut formula is principal in both premises,  neither premise has been derived by a geometric rule and we proceed as for $\mathsf{Gi}$.
\end{enumerate}
 \end{proof}

\section{Singular geometric theories}

To prove interpolation in extensions of first-order logic, the class of geometric rules seems too large. Thus, we restrict our attention to a proper sub-class of it and we introduce the class of singular geometric theories. In the next section we will prove (Lemma \ref{th:maehara_singular}) that Maehara's lemma holds for singular geometric extensions of first-order logic. 

A \emph{singular geometric axiom} is a geometric axiom with at most one non-logical predicate and no constant occurring essentially. 
A \emph{singular geometric theory} is a theory containing only singular geometric axioms. In sequent calculus a singular geometric theory can be formulated by extending $\mathsf{G}$ with finitely many geometric rules of form:
 $$
  \infer[\infrule{R}]{P_1 , \dots , P_n , \Gamma \Rightarrow \Delta}{\mathbf{Q}^*_1 , P_1 , \dots , P_n , \Gamma \Rightarrow \Delta & \cdots & \mathbf{Q}^*_m , P_1 , \dots , P_n , \Gamma \Rightarrow \Delta}
  $$
  where no constant occurs essentially and that satisfy the following singularity condition:

     \begin{equation}\tag{$\star$}
     |\mathsf{Rel} (\mathbf{Q}^*_1 , \dots , \mathbf{Q}^*_m , P_1 , \dots , P_n)| \leq 1 
     \end{equation}
 \noindent    Singular geometric axioms are ubiquitous in mathematics. Here, for example, is an incomplete list of singular geometric axioms for a binary relation $R$ (the list is partly taken from \cite[p. 48-50]{Casari2006}).
   
   \
   
   \begin{tabular}{ll}
    $R$ is reflexive & $\forall x (\top \to x R x)$\\
    $R$ is irreflexive & $\forall x (x R x \to \bot)$\\
    $R$ is transitive & $\forall x \forall y \forall z (x R y \wedge y R z \to x R z)$\\
    $R$ is intransitive & $\forall x \forall y \forall z (x R y \wedge y R z \wedge x R z \to \bot)$\\
    $R$ is co-transitive & $\forall x \forall y \forall z(x R y\to xR z \vee z Ry)$\\
    $R$ is splitting & $\forall x \forall y \forall z (x R y \to xR z \vee y R z)$\\
     $R$ is symmetric & $\forall x \forall y (x R y \to y R x)$\\
    $R$ is asymmetric & $\forall x \forall y (x R y \wedge y R x \to \bot)$\\
     \end{tabular}
    
    \begin{tabular}{ll}
    $R$ is anti-symmetric & $\forall x \forall y (x R y \wedge y R x \to x = y)$\\
    $R$ is trichotomy & $\forall x \forall y (\top \to x = y \vee x R y \vee y R x)$\\
    
    $R$ is linear & $\forall x \forall y (\top \to x R y \vee y R x)$\\
    $R$ is Euclidean & $\forall x \forall y \forall z (x R z \wedge y R z \to x R y)$\\
    $R$ is left-unique & $\forall x \forall y \forall z (x R z \wedge y R z \to x = y)$\\
  
   $R$ is right-unique & $\forall x \forall y \forall z (z R x \wedge z R y \to x = y)$\\
    $R$ is connected & $\forall x \forall y \forall z (x R y \wedge x R z \to y R z \vee z R y)$\\
    $R$ is nilpotent & $\forall x \forall y \forall z (x R z \wedge z R y \to \bot)$\\
    $R$ is a left ideal & $\forall x \forall y \forall z (x R y \to x R z)$\\
     
    $R$ is a right ideal & $\forall x \forall y \forall z (x R y \to z R y)$\\
    $R$ is rectangular & $\forall x \forall y \forall z \forall v (x R z \wedge v R y \to  x R y)$\\
    $R$ is dense & $\forall x \forall y (x R y \to \exists z (x R z \wedge z R y))$\\
    $R$ is total & $\forall x \exists y (\top \to x R y)$\\
    $R$ is confluent & $\forall x \forall y \forall z (x R y \wedge x R z \to \exists u (y R u \wedge z R u))$\\
    $R$ is left-oriented & $\forall x \forall y (\top \to \exists z (z R x \wedge z R y))$\\
    $R$ is right-oriented & $\forall x \forall y (\top \to \exists z (x R z \wedge y R z))$\\
   \end{tabular}

   \
  
  It is evident that a number of important classical and intuitionistic mathematical theories are singular geometric. Regarding the classical ones, the theory of partial orders ($R$ is reflexive, transitive and anti-symmetric), the theory of linear orders ($R$ is a linear partial order), as well as the theories of strict partial orders ($R$ is irreflexive and transitive) and strict linear orders ($R$ is a trichotomic strict partial order) are singular geometric. Constructive singular geometric theories, on the other hand, include von Plato's theories of positive partial orders \cite{vonPlato2001} ($R$ is irreflexive and co-transitive) and positive linear orders ($R$ is an asymmetric positive partial order), as well as the theory of apartness ($R$ is irreflexive and splitting). Also the theory of equivalence relations ($R$ is reflexive, transitive and symmetric) falls within the class of singular geometric. Finally, the fact the a relation $R$ is functional (total and right-unique) can be axiomatized using singular geometric axioms. Singular geometric axioms are important in logic, too. Specifically, the axioms of identity are singular geometric. 
   
   \
   
   \begin{tabular}{ll}
    $=$ is reflexive & $\forall x (x = x)$\\
    $=$ satisfies the indescernibility of identicals & $\forall x \forall y (x = y \wedge P [{}_{z}^{x}] \to P [{}_{z}^{y}])$\\    
   \end{tabular}

   \

   Notice that  the indiscernibility of identicals satisfies the singularity condition ($\star$)  because  identity is a logical predicate.  Hence, first-order logic with identity is a singular geometric theory.  
   
   Cut elimination for singular geometric rules clearly follows from cut elimination for geometric rules. More precisely, let $\mathsf{G^s}$ be any extension of $\mathsf{G}$ with singular geometric rules. Then:
   
    \begin{cor}\label{prop:structural_properties_singular_geometric}
All derivability properties expressed in Lemma \ref{th:substitution_general}, Theorem \ref{prop:structural_properties_geometric_cl} and Theorem \ref{prop:structural_properties_geometric_in} hold for $\mathsf{G^s}$.
 \end{cor}

 \begin{proof}
 Straightforward, since all singular geometric rules are geometric.  
 \end{proof}
  
 
  \section{Interpolation with singular geometric rules}
  
  The standard proof of interpolation in sequent calculi rests on a result due to Maehara which appeared (in Japanese) in \cite{Maehara1960} and was later made available to international readership by Takeuti in his \cite{Takeuti1987}. While interpolation is a result about logic, regardless the formal system (sequent calculus, natural deduction, axiom system, etc), Maehara's lemma is a \textquotedblleft sequent-calculus version\textquotedblright{} of interpolation. 
  Although originally Maehara proved his lemma for $\mathsf{LK}$, it is easy to adapt the proof so that it holds also in $\mathsf{G}$ (cf. \cite[\S 4.4]{TroelstraSchwichtenberg2000}). We recall from \cite{TroelstraSchwichtenberg2000} some basic definitions. 
  \begin{defn}[partition, split-interpolant]\label{partition}
  A \emph{partition of a sequent $\GSD$} is an expression $\Gamma_1 \prt \Gamma_2 \Rightarrow \Delta_1 \prt \Delta_2$, where $\Gamma = \Gamma_1 , \Gamma_2$ and $\Delta = \Delta_1 ,\Delta_2$ (for = the multiset-identity). 
  A \emph{split-interpolant} of a partition $\Gamma_1 \prt  \Gamma_2 \Rightarrow \Delta_1  \prt  \Delta_2$ is a formula $C$ such that:
 \begin{enumerate}
  \item[I] \qquad$\vdash \Gamma_1 \Rightarrow \Delta_1 , C$
  \item[II] \qquad$\vdash C , \Gamma_2 \Rightarrow \Delta_2$
  \item[III] \qquad$\mathcal{L}(C) \subseteq \mathcal{L} (\Gamma_1 , \Delta_1) \cap \mathcal{L} (\Gamma_2 , \Delta_2)$
 \end{enumerate} 
   We use $\Gamma_1 \prt \Gamma_2 \xRightarrow{C} \Delta_1 \prt \Delta_2$ to indicate that $C$ is a split-interpolant for $\Gamma_1 \prt \Gamma_2 \Rightarrow \Delta_1 \prt \Delta_2$. 
    \end{defn}
    Moreover, we say that a $C$ satisfying conditions (I) and (II) satisfies the derivability conditions for being a split-interpolant for the partition $\Gamma_1 \prt \Gamma_2 \Rightarrow \Delta_1 \prt \Delta_2$, whereas if $C$ satisfies (III) we say that it satisfies the language condition for being a split-interpolant for the same partition. 
  
 \begin{lemma}[Maehara]\label{th:maehara_logic}
 In $\mathsf{Gc}$ every partition  $
\Gamma_1\prt\Gamma_2\To\Delta_1\prt\Delta_2
 $ of a derivable sequent $\Gamma\To\Delta$ has a split-interpolant. In $\mathsf{Gi}$ every partition  $\Gamma_1\prt\Gamma_2\To\prt A $ of a derivable sequent $\Gamma\To A$ has a split-interpolant.
  
 \end{lemma}
 
The proof is by induction on the height $h$ of the derivation. If $h = 0$ then $\Gamma \Rightarrow \Delta$ is an initial sequent or a conclusion of a 0-premise rule and the proof is as in \cite{TroelstraSchwichtenberg2000}.\footnote{Notice, however, that the proof given in \cite{TroelstraSchwichtenberg2000} contains a misprint and the split-interpolant for the partition of the initial sequent $\Gamma_1 , P \prt \Gamma_2 \Rightarrow \Delta_1 , P \prt \Delta_2$ (their notation adjusted to ours) is $\bot$, and not $\bot \to \bot$ as stated in \cite[p.117]{TroelstraSchwichtenberg2000}.}  If $h=n+1$ one uses as induction hypothesis  the fact that any partition of the premises of a rule $R$ has a split-interpolant. For a detailed proof the reader is again referred to \cite{TroelstraSchwichtenberg2000}.

From Maehara's lemma it is immediate to prove Craig's interpolation theorem.
\begin{thm}[Craig]\label{Craig}
 If $A \Rightarrow B$ is derivable in $\mathsf{G}$ then there exists a $C$ such that $\vdash A \Rightarrow C$ and $\vdash C \Rightarrow B$ and $\mathcal{L}(C) \subseteq \mathcal{L} (A) \cap \mathcal{L} (B)$. 
\end{thm}
\begin{proof}
   Let $A \Rightarrow B$ be derivable in $\mathsf{G}$ and consider the partition $A \prt \varnothing \Rightarrow \varnothing \prt B$ of $A \Rightarrow B$. By Lemma \ref{th:maehara_logic}, this partition has a split-interpolant, namely there exists a $C$ such that $A \prt \varnothing \xRightarrow{C} \varnothing \prt B$. Hence $\vdash A \Rightarrow C$ and $\vdash C \Rightarrow B$ and $\mathcal{L}(C) \subseteq \mathcal{L} (A) \cap \mathcal{L} (B)$ by Definition \ref{partition}. 
  \end{proof}
Of any calculus for which Theorem \ref{Craig} holds, we say that it has the interpolation property. Now we extend Lemma \ref{th:maehara_logic} to extensions of $\mathsf{G}$ with singular geometric rules. 
 
In the proof of Lemma \ref{th:maehara_singular}  we shall only consider singular geometric rules where each $\mathbf{Q}^*_i$ is a single atom $Q^*_i$. More precisely, we consider singular geometric rules of the form 
$$
  \infer[\infrule{R}]{P_1 , \dots , P_n , \Gamma \Rightarrow \Delta}{Q_1^* , P_1 , \dots , P_n , \Gamma \Rightarrow \Delta & \cdots & Q_m^* , P_1 , \dots , P_n , \Gamma \Rightarrow \Delta}
$$
where $\Delta$ consists of exactly one formula in $\mathsf{Gi}$. This allows some notational simplification and will significantly improve the readability of the proof. It does not impair the generality of the result.  

 \begin{lemma}\label{th:maehara_singular}
 In $\mathsf{Gc^s}$ every partition $\Gamma_1 \prt \Gamma_2 \To \Delta_1 \prt \Delta_2$ of a derivable sequent $\Gamma\To\Delta$ has a split-interpolant. In $\mathsf{Gi^s}$ every partition $\Gamma_1 \prt \Gamma_2 \To  \prt A$ of a derivable sequent $\Gamma\To A$ has a split-interpolant.
 \end{lemma}

 \begin{proof}
 The proof extends that of Lemma \ref{th:maehara_logic}. Let $R$ be a singular geometric rule with $m$ premises and let $\Pi , \GSD$ be its conclusion, where $\Pi$ is the multiset  $P_1,\dots,P_n$ of the atomic principal formulas of $R$, if any. We consider the following generic partition of the conclusion:
 $$
 \Pi_1,\Gamma_1\prt\Pi_2,\Gamma_2\To\Delta_1\prt\Delta_2
 $$
 where $\Pi_1,\Pi_2=\Pi$ and $\Gamma_1,\Gamma_2=\Gamma$ and $\Delta_1,\Delta_2=\Delta$, and where $\d_1=\varnothing$ and $\d_2=A$ for $\mathsf{Gi^s}$. Moreover, let $\Theta$ be the multiset $Q_1^*,\dots,Q_m^*$ of active formulas of $R$, if any. We organize the proof in three exhaustive cases:
  \begin{enumerate}
\item\label{case1} $\mathsf{Rel}(\Theta,\Pi)\subseteq\mathsf{Rel}(\Pi_1,\g_1,\d_1)$;
\item\label{case2} $\mathsf{Rel}(\Theta,\Pi)\subseteq\mathsf{Rel}(\Pi_2,\g_2,\d_2)$;
\item\label{case3} $\mathsf{Rel}(\Theta,\Pi)\not\subseteq\mathsf{Rel}(\Pi,\g,\d)$.
 \end{enumerate}
Observe that these three cases are exhaustive since singular geometric rules have at most one non-logical predicate in their principal and active  formulas and, therefore,  when Case 3 does not hold at least one of Cases 1 and 2 holds. 
We give a proof of the three  cases for $\mathsf{Gc}$, and then we show the modifications needed for $\mathsf{Gi}$.

\

\emph{Case \ref{case1} for $\mathsf{Gc^s}$}. If $R$ is an $m$-premise(s) rule for $m\geq1$, then  by the inductive hypothesis (IH) every partition of each of the $m$ premises of $R$ has a split-interpolant. In particular,   for each $k\in\{1,\dots,m\}$, there is a $C_k$  such that:
 \begin{itemize}
 \item[(I$_k$)] $\vdash  Q^*_k ,
 \Pi_1,\Pi_2, \Gamma_1 \Rightarrow \Delta_1 , C_k $
  \item[(II$_k$)] $\vdash  C_k , \Gamma_2 \Rightarrow \Delta_2 $
  \item[(III$_k$)] $\L (C_k) \subseteq \L (Q^*_k , 
  \Pi_1,\Pi_2, \Gamma_1 , \Delta_1) \cap \L (\Gamma_2 , \Delta_2)$
   \end{itemize}
   If, instead, $R$ is a 0-premise rule then (I$_1$), (II$_1$), and (III$_1$) hold trivially for $C_1\equiv\bot$.
   
We start by assuming that $\Pi_2$ is the non-empty multiset $P_{i_{j+1}} , \dots , P_{i_{n}}$, and then we show the modifications needed when $\Pi_2=\varnothing$. Consider now the following derivation $\mathcal{D}_1$, where the topmost sequents are derivable by (I$_1$)\,-\,(I$_m$):
\begin{small}
 \begin{equation}\label{firstcase1}
 \infer[\infrule{R\to}]{\Pi_1,\Gamma_1\Rightarrow\Delta_1, \bigwedge\Pi_2 \rightarrow \bigvee_{i=1}^m C_i}{
\infer=[\infrule{L\wedge}]{\bigwedge\Pi_2,\Pi_1,\Gamma_1\Rightarrow\Delta_1,  \bigvee_{i=1}^m C_i
}
{\infer[\infrule R]{\Pi_2,\Pi_1,\Gamma_1\Rightarrow\Delta_1,  \bigvee_{i=1}^m C_i
}{
\infer=[\infrule R\lor]{Q_1^*,\Pi_2,\Pi_1,\Gamma_1\Rightarrow\Delta_1,  \bigvee_{i=1}^m C_i}{
\infer=[\infrule Wkn]{Q_1^*,\Pi_2,\Pi_1,\Gamma_1\Rightarrow\Delta_1,C_1,\dots,C_m}{Q_1^*,\Pi_2,\Pi_1,\Gamma_1\Rightarrow\Delta_1,C_1}}&
\deduce{\phantom{a}}{\dots}&
\infer=[\infrule R\lor ]{Q_m^*,\Pi_2,\Pi_1,\Gamma_1\Rightarrow\Delta_1,  \bigvee_{i=1}^m C_i}{
\infer=[\infrule Wkn]{Q_m^*,\Pi_2,\Pi_1,\Gamma_1\Rightarrow\Delta_1,C_1,\dots,C_m}{Q_m^*,\Pi_2,\Pi_1,\Gamma_1\Rightarrow\Delta_1,C_m}}
}
}
}
\end{equation}
\end{small}
Notice that the application of $R$ is legitimate because by assumption $R$ is applicable to $Q_i^*,\Pi,\Gamma\To\Delta$ and none of the \emph{eigenvariables} of the $Q_i^*$'s can occur free in some $C_k$, since $\mathcal L(C_k)\subseteq\mathcal L(\Gamma_2,\d_2)$.  Notice also that in some particular case the double-line stands for the empty sequence of instances,  e.g.,  the steps by $R\lor$ when $R$ is a 0- or 1-premise rule.

Consider another derivation $\mathcal{D}_2$, where the left-topmost sequents are initial sequents since $\Pi_2=P_{i_{j+1}} , \dots , P_{i_{n}}$ and the right-topmost ones are derivable by (II$_1$)--(II$_m$):
{\small 
\begin{equation}\label{secondcase1}
 \infer[\infrule{L\to}]{
 \bigwedge\Pi_2 \to  \bigvee_{i=1}^m C_i ,  \Pi_2 ,\Gamma_2\Rightarrow\Delta_2
 }
 {
  \infer=[\infrule{R\wedge}]
  {
  \Pi_2 ,\Gamma_2\Rightarrow\Delta_2 , \bigwedge\Pi_2
  }
 {
  \Pi_2,\Gamma_2\Rightarrow\Delta_2 , P_{i_{j+1}} & \cdots & \Pi_2,\Gamma_2\Rightarrow\Delta_2 , P_{i_{n}}
 }
&
\infer=[\infrule{Wkn}]
{
 \bigvee_{i=1}^m C_i ,  \Pi_2 ,\Gamma_2\Rightarrow\Delta_2
}
{
\infer=[\infrule{L\vee}]
{
 \bigvee_{i=1}^m C_i ,  \Gamma_2\Rightarrow\Delta_2
}
{
C_1 ,  \Gamma_2\Rightarrow\Delta_2
&
\cdots
&
C_m ,  \Gamma_2\Rightarrow\Delta_2
}
}
 }
 \end{equation}}
 
 When $\Pi_2=\varnothing$ we modify $\mathcal{D}_1$ by using left weakening   instead of $L\wedge$ to add $\bigwedge\Pi_2$ --i.e., $\top$ -- to the antecedent, and we modify $\mathcal{D}_2$ by deriving the conclusion of $R\wedge$ by an instance of $R\top$ instead of by instances of $R\wedge$.

 Let $t_1 , \dots , t_\ell$ be all terms such that $t_1 , \dots , t_\ell \in \mathsf{Ter} (\bigwedge\Pi_2\to  \bigvee_{i=1}^m C_i)$ and $(\bullet)\;t_1 , \dots , t_\ell \notin \mathsf{Ter}(\Pi_1 , \Gamma_1 , \Delta_1) \cap \mathsf{Ter} (\Pi_2 , \Gamma_2 , \Delta_2)$. We use $\overline{t}$ to denote $t_1,\dots,t_\ell$. We show that 
 $$
 (\ddag)\qquad t_1 , \dots , t_\ell \notin \mathsf{Ter} ( \Pi_1 , \Gamma_1 , \Delta_1)
 $$
 For each $k\leq m$, (III$_k$) entails that $\mathsf{Ter}(C_k)\subseteq \mathsf{Ter}(\Gamma_2,\Delta_2)$. Hence  $\mathsf{Ter}(\bigwedge\Pi_2\to  \bigvee_{i=1}^m C_i)\subseteq \mathsf{Ter}(\Pi_2,\Gamma_2,\Delta_2)$. By this and ($\bullet$) we immediately get that $(\ddag)$ holds.
%

Let now $\bar{z}$ be  variables $z_1 , \dots , z_\ell$ not occurring in $\mathcal{D}_1$ and $\mathcal{D}_2$. Lemma \ref{th:substitution_general} applied to  $\mathcal D_1$ shows that:
$$
 \vdash \Pi_1, \Gamma_1 \Rightarrow \Delta_1 , ( \bigwedge\Pi_2\to  \bigvee_{i=1}^m C_i ) [\,{}_{\bar{t}}^{\bar{z}}\,]
$$
Here $(\ddag$) ensures that the substitution $[\,{}_{\bar{t}}^{\bar{z}}\,]$ has no effect on $\Pi_1, \Gamma_1,\Delta_1$. By $\ell$ applications of $R\forall$ to the derivable sequent above we obtain:
$$
\textnormal{(I}_C) \qquad \vdash \Pi_1 , \Gamma_1 \Rightarrow \Delta_1 , \forall \bar{z}(( \bigwedge\Pi_2\to  \bigvee_{i=1}^m C_i ) [\,{}_{\bar{t}}^{\bar{z}}\,])
$$
Moreover, by applying $\ell$ instances of left weakening 
and then $\ell$ instances of $L\forall$ to the conclusion of $\mathcal D_2$ we obtain:

%
$$\textnormal{(II}_C) \qquad\vdash \forall \bar{z}(( \bigwedge\Pi_2\to  \bigvee_{i=1}^m C_i )[\,{}_{\bar{t}}^{\bar{z}}\,]) ,\Pi_2
 , \Gamma_2   \Rightarrow \Delta_2 
$$
%
%
Let $C$ be $\forall \bar{z}(( \bigwedge\Pi_2\to  \bigvee_{i=1}^m C_i )[\,{}_{\bar{t}}^{\bar{z}}\,])$. By (I$_C$) and (II$_C$), we have established that $C$ satisfies the derivability conditions for being a split-interpolant of the given partition. We now show that it also satisfies the language condition, namely:
$$
\textnormal{(III}_C) \quad \L (C) \subseteq \L (\Pi_1, \Gamma_1, \Delta_1) \cap \L(\Pi_2 , \Gamma_2, \Delta_2)
$$
First, if $s$ is a term in $\mathsf{Ter} (C)$, it is a term occurring in  $\bigwedge\Pi_2\to  \bigvee_{i=1}^m C_i$ that is not in the list $\bar{t}$. By $(\bullet)$, we have: 
 $$\text{(III}.1_C)\quad s \in \mathsf{Ter}(\Pi_1 , \Gamma_1,\Delta_1)\cap\mathsf{Ter}(\Pi_2 ,\Gamma_2,\Delta_2)$$ 
 Next, we show that: 
$$(\textnormal{III}.2_C)\qquad\mathsf{Rel}(C)\subseteq\mathsf{Rel}(\Pi_1,\Gamma_1,\Delta_1)\cap\mathsf{Rel}(\Pi_2,\Gamma_2,\Delta_2)
$$
By assumption,  we are in Case \ref{case1}, i.e., $ \mathsf{Rel}(\Theta,\Pi)\subseteq\mathsf{Rel}(\Pi_1,\g_1,\d_1)$. The following set-theoretic reasoning shows that ($\textnormal{III}.2_C$) holds:

$\begin{array}{lc}
\mathsf{Rel}(C)&\stackrel{\textnormal{III}_k}\subseteq\\\noalign{\medskip}
\mathsf{Rel}(\Pi_2)\cup(\mathsf{Rel}(\Theta ,   \Pi_1,\Pi_2, \Gamma_1 , \Delta_1) \cap \mathsf{Rel} (\Gamma_2 , \Delta_2))&\stackrel{\textnormal{distrib.}}=\\\noalign{\medskip}
\mathsf{Rel}(\Theta,  \Pi_1,\Pi_2, \Gamma_1 , \Delta_1) \cap \mathsf{Rel} (\Pi_2,\Gamma_2 , \Delta_2)&\stackrel{\text{Case }\ref{case1}}=\\\noalign{\medskip}
 \mathsf{Rel}( \Pi_1, \Gamma_1 , \Delta_1) \cap \mathsf{Rel} (\Pi_2,\Gamma_2 , \Delta_2)
\end{array}$

\

%
We conclude that: 
$$
\framebox{
    \infer
  {
   \Pi_1 , \Gamma_1 \prt \Pi_2 , \Gamma_2 \xRightarrow{\forall \bar{z} (( \bigwedge\Pi_2\to\bigvee_{i=1}^mC_i)[\,{}_{\bar{t}}^{\bar{z}}\,])} \Delta_1 \prt \Delta_2
  }
  {
   Q^*_1 ,  \Pi_1 , \Pi_2 , \Gamma_1 \prt  \Gamma_2 \xRightarrow{C_1} \Delta_1 \prt \Delta_2
   &
   \cdots
   &
    Q^*_m ,  \Pi_1 , \Pi_2 ,   \Gamma_1 \prt \Gamma_2 \xRightarrow{C_m} \Delta_1 \prt \Delta_2
  }
}
  $$
Observe that when $\Pi_2=\varnothing$ the split-interpolant of the conclusion can be simplified as follows:
$$
\framebox{
\deduce{\forall \bar{z} (( \bigvee_{i=1}^mC_i)[\,{}_{\bar{t}}^{\bar{z}}\,])}{}
}
$$

\

\emph{Case \ref{case2} for $\mathsf{Gc^s}$}. The proof differs substantially from that of Case \ref{case1} only as far as the derivability conditions are concerned. Thus, we give a detailed analysis of these and leave to the reader the task to check that also the language condition is satisfied. By IH every partition of each premise of an $m$-premises  ($m\geq 1$) rule $R$  has a split-interpolant. In particular, for all $k\in\{1,\dots,m\}$, there is a $C_k$ such that:
  
 \begin{itemize}
 \item[(I$_k$)] $\vdash \Gamma_1 \Rightarrow \Delta_1 , C_k $
  \item[(II$_k$)] $\vdash  C_k ,  Q^*_k ,\Pi_1,\Pi_2
  , \Gamma_2 \Rightarrow \Delta_2 $
  \item[(III$_k$)] $\L (C_k) \subseteq \L (\Gamma_1 , \Delta_1) \cap \L (Q^*_k , \Pi_1,\Pi_2
   , \Gamma_2 , \Delta_2)$
 
   \end{itemize}
In case $R$ is a 0-premise rule, (I$_1$), (II$_1$), and (III$_1$) hold by imposing $C_1\equiv\top$.
   
 Let $\mathcal{D}_1$ be the following derivation, where the topmost sequents are  derivable by (II$_1$)\,-\,(II$_m$): 
\begin{small}
 \begin{equation}\label{firstcase2}
  \infer=[\infrule{L\wedge}]
  {
  \bigwedge_{i=1}^m C_i\wedge\bigwedge\Pi_1, \Pi_2 , \Gamma_2 \Rightarrow \Delta_2
  }
  {
  \infer[\infrule{R}]
  {
   C_1 , \dots , C_m , \Pi_1  , \Pi_2 , \Gamma_2 \Rightarrow \Delta_2
  }
  {
   \infer=[\infrule{Wkn}]
   {
  C_1 ,  \dots , C_m , Q_1^*  , \Pi_1 , \Pi_2 , \Gamma_2 \Rightarrow \Delta_2
   }
   {
  C_1 , Q_1^*   ,  \Pi_1 , \Pi_2 , \Gamma_2 \Rightarrow \Delta_2
    }
  &
  \cdots
  &
  \infer=[\infrule{Wkn}]
  {
   C_1 , \dots , C_m , Q_m^* , \Pi_1 , \Pi_2 , \Gamma_2 \Rightarrow \Delta_2
  }
  {
 C_m , Q_m^* ,  \Pi_1 , \Pi_2 , \Gamma_2 \Rightarrow \Delta_2
  }
  }
  }
  \end{equation}
  \end{small}
Consider now another derivation $\mathcal{D}_2$ where the left topmost sequents are derivable by (I$_1$)\,-\,(I$_m$) and the right ones are initial sequents (we take  $P_{i_1},\dots,P_{i_j}= \Pi_1$ if $\Pi_1\neq\varnothing$, else, as we did in (\ref{secondcase1}), we derive the conclusion of the right top-most instance(s) of $R\wedge$ by $R\top$):
%

 {\small
 \begin{equation}\label{secondcase2}
  \infer[\infrule{R\wedge}]
  {
   \Pi_1 , \Gamma_1 \Rightarrow \Delta_1 ,   \bigwedge_{i=1}^m C_i\wedge\bigwedge\Pi_1
  }
  {
 \infer=[\infrule{Wkn}]
 {
 \Pi_1 , \Gamma_1 \Rightarrow \Delta_1 , \bigwedge_{i=1}^m C_i
 }
 {
 \infer=[\infrule{R\wedge}]
 {
 \Gamma_1 \Rightarrow \Delta_1 , \bigwedge_{i=1}^m C_i
 }
 {
 \Gamma_1 \Rightarrow \Delta_1 , C_1 
 &
 \cdots
 &
 \Gamma_1 \Rightarrow \Delta_1 , C_m
 }
 }
   &
\infer=[\infrule{R\wedge}]{  \Pi_1 , \Gamma_1 \Rightarrow \Delta_1 , \bigwedge\Pi_1}
{\Pi_1 , \Gamma_1 \Rightarrow \Delta_1 , P_{i_{1}} & \cdots & \Pi_1 , \Gamma_1 \Rightarrow \Delta_1 , P_{i_{j}}}
  }
  \end{equation}  }
%
 Let $\bar{t}$ be  all terms $t_1 , \dots , t_\ell$ such that $t_1 , \dots , t_\ell \in \mathsf{Ter}( \bigwedge_{i=1}^m C_i\wedge\bigwedge\Pi_1)$ and $t_1 , \dots , t_\ell \notin \mathsf{Ter} (\Pi_1 , \Gamma_1 , \Delta_1) \cap \mathsf{Ter} (\Pi_2 , \Gamma_2 , \Delta_2)$. As in the previous case, it is easy to show that: 
 $$
 (\ddag)\qquad t_1 , \dots , t_\ell \notin \mathsf{Ter}(\Pi_2,\Gamma_2,\Delta_2)
 $$
 Moreover let $\bar{z}$ be variables $z_1 , \dots , z_\ell$ new to $\mathcal{D}_1$ and $\mathcal{D}_2$. We reason analogously to  the previous case to obtain:
 $$
 \textnormal{(I}_C) \qquad \vdash\Pi_1 , \Gamma_1 \Rightarrow \Delta_1 , \exists \bar{z}((\bigwedge_{i=1}^m C_i\wedge\bigwedge\Pi_1)[\,{}_{\bar{t}}^{\bar{z}}\,]) 
 $$
 As above, thanks to ($\ddag$), we also obtain:
 $$
 \textnormal{(II}_C) \qquad \vdash \exists \bar{z}((\bigwedge_{i=1}^m C_i\wedge\bigwedge\Pi_1)[\,{}_{\bar{t}}^{\bar{z}}\,]),\Pi_2 , \Gamma_2 \Rightarrow \Delta_2
 $$
  Let $C$ be $\exists \bar{z}((\bigwedge_{i=1}^m C_i\wedge\bigwedge\Pi_1)[\,{}_{\bar{t}}^{\bar{z}}\,])$. Given that $\mathsf{Rel}(\Theta,\Pi)\subseteq\mathsf{Rel}(\Pi_2,\g_2,\d_2)$,  and given that we have quantified away all terms in $\bar{t}$, we have:  
$$
\textnormal{(III}_C) \qquad \L (C) \subseteq \L (\Pi_1, \Gamma_1 , \Delta_1) \cap \L (\Pi_2, \Gamma_2 , \Delta_2)
$$
We conclude that $C$ is a split-interpolant of the given partition.  
  $$
\framebox{
  \infer
  {
   \Pi_1 , \Gamma_1 \prt \Pi_2 , \Gamma_2 \xRightarrow{\exists \bar{z}((\bigwedge_{i=1}^m C_i\wedge\bigwedge\Pi_1)[\,{}_{\bar{t}}^{\bar{z}}\,])} \Delta_1 \prt \Delta_2
  }
  {
   \Gamma_1 \prt Q^*_1 ,  \Pi_1 , \Pi_2 , \Gamma_2 \xRightarrow{C_1} \Delta_1 \prt \Delta_2
   &
   \dots
   &
    \Gamma_1 \prt Q^*_m ,  \Pi_1 , \Pi_2  ,  \Gamma_2 \xRightarrow{C_m} \Delta_1 \prt \Delta_2
  }
}
$$
As for the previous case, when $\Pi_1=\varnothing$ we  have a simpler split-interpolant of the conclusion:
$$
\framebox{
\deduce{\exists \bar{z}((\bigwedge_{i=1}^m C_i)[\,{}_{\bar{t}}^{\bar{z}}\,])}{}
}
$$

\
\emph{Case \ref{case3} for $\mathsf{Gc^s}$.} We can proceed either as in Case \ref{case1} or as in Case \ref{case2}.
If we proceed as in Case \ref{case1}, we obtain the following split-interpolant:

$$
\framebox{
    \infer
  {
   \Pi_1 , \Gamma_1 \prt \Pi_2 , \Gamma_2 \xRightarrow{\forall \bar{z} (( \bigwedge\Pi_2\to\bigvee_{i=1}^mC_i)[\,{}_{\bar{t}}^{\bar{z}}\,])} \Delta_1 \prt \Delta_2
  }
  {
   Q^*_1 ,  \Pi_1 , \Pi_2 , \Gamma_1 \prt  \Gamma_2 \xRightarrow{C_1} \Delta_1 \prt \Delta_2
   &
   \cdots
   &
    Q^*_m ,  \Pi_1 , \Pi_2 ,   \Gamma_1 \prt \Gamma_2 \xRightarrow{C_m} \Delta_1 \prt \Delta_2
  }
}
  $$
  
  The proof that the formula $C$ presented above is the split-interpolant of the conclusion is exactly as for Case \ref{case1}, save for the relational part (III.2$_C$) of the language condition. In this case we are assuming that $\mathsf{Rel}(\Theta,\Pi)\not\subseteq\mathsf{Rel}(\Pi,\g,\d)$. This  immediately implies 
  $$
  (+)\quad \mathsf{Rel}(\Pi_1,\Pi_2)=\varnothing
  $$
  and, together with the fact that $|\mathsf{Rel}(\Theta)|\leq 1$, it implies
  $$
  (++)\quad  \mathsf{Rel}(\Theta)\cap\mathsf{Rel}(\Pi_2,\g_2,\d_2)=\varnothing
  $$
  Hence, we can show that (III.2$_C$) holds via the following set-theoretic reasoning
  
  $\begin{array}{lc}
\mathsf{Rel}(C)&\stackrel{\textnormal{III}_k}\subseteq\\\noalign{\medskip}
\mathsf{Rel}(\Pi_2)\cup(\mathsf{Rel}(\Theta ,   \Pi_1,\Pi_2, \Gamma_1 , \Delta_1) \cap \mathsf{Rel} (\Gamma_2 , \Delta_2))&\stackrel{\textnormal{distrib.}}=\\\noalign{\medskip}
\mathsf{Rel}(\Theta ,  \Pi_1,\Pi_2, \Gamma_1 , \Delta_1) \cap \mathsf{Rel} (\Pi_2,\Gamma_2 , \Delta_2)&\stackrel{(+),(++)}=\\\noalign{\medskip}
 \mathsf{Rel}(  \Gamma_1 , \Delta_1) \cap \mathsf{Rel} (\Gamma_2 , \Delta_2)
\end{array}$


\

\emph{Case \ref{case1}  for $\mathsf{Gi^s}$}. The proof is the same as for Case \ref{case1} in $\mathsf{Gc^s}$ (with $\Delta_1=\varnothing$ and $\Delta_2= A$) save for the derivations $\D_1$  and $\D_2$  presented in (\ref{firstcase1}) and (\ref{secondcase1}) that are not  $\mathsf{Gi^s}$-derivations. It is immediate to see that we can obtain a $\mathsf{Gi^s}$-derivation from the derivation in (\ref{firstcase1}) by simply omitting the instances of weakening and applying directly instances of $R\lor$ to the leaves. On the other hand, the derivation $\D_2$ presented in (\ref{secondcase1}) becomes a  $\mathsf{Gi^s}$-derivation by simply dropping the singleton multiset  $\Delta_2$ from the left top-most sequents and then adding an instance of weakening on the left premise of $L\to$.

\

\emph{Case \ref{case2} for $\mathsf{Gi^s}$}. The proof is the same as for Case \ref{case2} in $\mathsf{Gc^s}$, since the derivations presented in (\ref{firstcase2}) and (\ref{secondcase2}) are $\mathsf{Gi^s}$-derivation when $\Delta_1=\varnothing$ and $\Delta_2=A$.

\

\emph{Case \ref{case3}  for $\mathsf{Gi^s}$}. We may proceed  as for  Case \ref{case1} for $\mathsf{Gi^s}$ save for the relational part (III.2$_C$) of the language condition  where we reason as in Case  \ref{case3} for $\mathsf{Gc^s}$.
  \end{proof}
       
From Lemma \ref{th:maehara_singular} it is immediate to conclude that singular geometric extensions of classical and intuitionistic logic satisfy the interpolation theorem, namely:
 \begin{thm}\label{CraigSingular}
 $\mathsf{G^s}$ has the interpolation property. 
 \end{thm}       

\section{Applications}
We now consider some corollaries of Theorem \ref{CraigSingular} in which the strategy for building interpolants provided in Lemma \ref{th:maehara_singular} is applied. Notice that in  the theories  considered in this section all contracted instances  are  admissible and, hence, we can ignore them, see the discussion after Definition \ref{closurecondition}.


\subsection{First-order logic with identity}

We start with first-order logic with identity. Recall that a cut-free calculus for classical first-order logic with identity has been presented in \cite{NegrivonPlato1998} by adding on top of $\mathsf{Gc}$ the rules \emph{Ref}$_=$ and \emph{Repl}$_=$ corresponding to the reflexivity of $=$ and Leibniz's principle of indescernibility of identicals, respectively. In intuitionistic theories, on the other hand, identity is often treated differently and we will provide a constructively more acceptable treatment of identity later in dealing with apartness. In general, however, nothing prevents us from building intuitionistic first-order logic with identity in a parallel fashion to the classical case. This is, for example, the route taken in \cite{TroelstraSchwichtenberg2000} and we will follow suit. More specifically, let $\mathsf{G}^=$ be $\mathsf{G} + \{ \textnormal{\emph{Ref}}_= , \textnormal{\emph{Repl}}_=\}$.
Notice that, since \emph{Ref}$_=$ and \emph{Repl}$_=$ are geometric rules, cut elimination holds in $\mathsf{Gi}^=$ in virtue of Theorem \ref{prop:structural_properties_geometric_in}. Moreover, since they are also singular geometric, it follows from our Theorem \ref{CraigSingular} that in $\mathsf{G}^=$ the interpolation property holds, i.e.
 \begin{cor}
  $\mathsf{G}^=$ has the interpolation property. 
 \end{cor}
\begin{proof}
We determine the split-interpolants as applications of the procedures given in the proof of Lemma \ref{th:maehara_singular}. The rule \emph{Ref}$_=$ can be treated as  an instance of Case \ref{case1} with $\Pi_2=\varnothing$ (obviously, it could also have been treated as an instance of Case \ref{case2}). Depending on whether both $c\in\mathsf{Ter}(C)$ and   $c\not\in\mathsf{Ter} (\Gamma_1 , \Delta_1)$ or  not, we have then, respectively: 
$$
\framebox{
  \infer
  {
    \Gamma_1 \prt  \Gamma_2 \xRightarrow{\forall z(C[{}_{s}^{z}])} \Delta_1 \prt \Delta_2
  }
  {
  s=s, \Gamma_1 \prt  \Gamma_2 \xRightarrow{C} \Delta_1 \prt \Delta_2
  }
  \qquad
   \infer
  {
    \Gamma_1 \prt  \Gamma_2 \xRightarrow{C} \Delta_1 \prt \Delta_2
  }
  {
  s=s, \Gamma_1 \prt  \Gamma_2 \xRightarrow{C} \Delta_1 \prt \Delta_2
  }
  }
  $$
For \emph{Repl}$_=$, there are four possible partitions of the conclusion: 
\begin{itemize}
 \item $s = t , P [{}_{x}^{s}] , \Gamma_1 \prt \Gamma_2 \Rightarrow \Delta_1 \prt \Delta_2$
\item $\Gamma_1 \prt s = t , P [{}_{x}^{s}] , \Gamma_2 \Rightarrow \Delta_1 \prt \Delta_2$
\item $P [{}_{x}^{s}] , \Gamma_1 \prt s = t , \Gamma_2 \Rightarrow \Delta_1 \prt \Delta_2$
\item $s = t , \Gamma_1 \prt P [{}_{x}^{s}] , \Gamma_2 \Rightarrow \Delta_1 \prt \Delta_2$
\end{itemize}
Accordingly, we need to consider four sub-cases. As in Case \ref{case1} of Lemma \ref{th:maehara_singular}, when $\Pi_2=\varnothing$, the interpolant for the first partition is as follows:
$$\framebox{
  \infer{s=t,P[{}_{x}^{s}], \Gamma_1 \prt \Gamma_2 \xRightarrow{C} \Delta_1 \prt \Delta_2 }{P[{}_{x}^{t}],s=t,P[{}_{x}^{s}], \Gamma_1 \prt \Gamma_2 \xRightarrow{C} \Delta_1 \prt \Delta_2}
  }$$
The interpolant for the second partition is obtained by reasoning as in Case \ref{case2} with $\Pi_1=\varnothing$ of Lemma \ref{th:maehara_singular}:
$$\framebox{
  \infer{\Gamma_1 \prt s=t,P[{}_{x}^{s}],\Gamma_2 \xRightarrow{C} \Delta_1 \prt \Delta_2 }{\Gamma_1 \prt P[{}_{x}^{t}],s=t,P[{}_{x}^{s}], \Gamma_2 \xRightarrow{C} \Delta_1 \prt \Delta_2}
  }$$
The interpolant for the third partition is found as in Case \ref{case1} of Lemma \ref{th:maehara_singular}, depending on whether $t\in\mathsf{Ter} (P[{}_{x}^{s}] , \Gamma_1 , \Delta_1)$ (left derivation in the box below) or  not (right derivation in the box below). 
    \begin{small}   $$
\framebox{
    \infer
  {
 P[{}_{x}^{s}] , \Gamma_1 \prt s=t , \Gamma_2 \xRightarrow{s = t  \to C} \Delta_1 \prt \Delta_2
  }
  {
   P[{}_{x}^{t}],s=t,P[{}_{x}^{s}] , \Gamma_1 \prt  \Gamma_2 \xRightarrow{C} \Delta_1 \prt \Delta_2
  }
 \quad   
   \infer
  {
 P[{}_{x}^{s}] , \Gamma_1 \prt s=t , \Gamma_2 \xRightarrow{\forall z (s = z \to C[{}_{t}^{z}])} \Delta_1 \prt \Delta_2
  }
  {
   P[{}_{x}^{t}],s=t,P[{}_{x}^{s}] , \Gamma_1 \prt  \Gamma_2 \xRightarrow{C} \Delta_1 \prt \Delta_2
  }
}$$\end{small}
  Lastly, the interpolant for the fourth partition is found as in Case \ref{case2} of Lemma \ref{th:maehara_singular},  depending on whether $t\in\mathsf{Ter} (P[{}_{x}^{s}] , \Gamma_2 , \Delta_2)$ or not:
  \begin{small}
  $$
\framebox{
 \infer
  {
s=t , \Gamma_1 \prt  P[{}_{x}^{s}], \Gamma_2 \xRightarrow{s = t  \wedge C} \Delta_1 \prt \Delta_2
  }
  {
   \Gamma_1 \prt   P[{}_{x}^{t}],s=t,P[{}_{x}^{s}] ,\Gamma_2 \xRightarrow{C} \Delta_1 \prt \Delta_2
  }
 \quad   \infer
  {
s=t , \Gamma_1 \prt  P[{}_{x}^{s}], \Gamma_2 \xRightarrow{\exists z (s = z \wedge C[{}_{t}^{z}])} \Delta_1 \prt \Delta_2
  }
  {
   \Gamma_1 \prt   P[{}_{x}^{t}],s=t,P[{}_{x}^{s}] ,\Gamma_2 \xRightarrow{C} \Delta_1 \prt \Delta_2
  }
  }
  $$
  \end{small}
   \end{proof}

\subsection{Equivalence relations}\label{EQ}

In a perfectly parallel fashion, we obtain the theory of equivalence relations by adding to $\mathsf{G}$ the rules corresponding to the reflexivity, transitivity and symmetry of a binary relation $\sim$. Thus, $\mathsf{EQ} = \mathsf{G} + \{\textnormal{\emph{Ref}}_\sim \, , \, \textnormal{\emph{Trans}}_\sim \, , \, \textnormal{\emph{Sym}}_\sim\}$. 

\begin{center}
 \begin{tabular}{cc}
$
\infer[\infrule{Ref_\sim}]{\Gamma \Rightarrow \Delta}{s \sim s , \Gamma \Rightarrow \Delta}
$
&
$
 \infer[\infrule{Trans_\sim}]{s \sim t , t \sim u , \Gamma \xRightarrow{} \Delta}{s \sim u , s \sim t , t \sim u , \Gamma \xRightarrow{} \Delta}
$
\\\\
\multicolumn{2}{c}{
$\infer[\infrule{Sym_\sim}]{s \sim t , \Gamma \xRightarrow{} \Delta}{t \sim s , s \sim t , \Gamma \xRightarrow{} \Delta}$
}
 \end{tabular}
\end{center}
From the fact that these rules are singular geometric, it follows that:

\begin{cor}
 $\mathsf{EQ}$ has the interpolation property. 
\end{cor}

\begin{proof} 
The case of \emph{Ref}$_\sim$  is like that for \emph{Ref}$_=$ in $\mathsf{G}^=$, the only difference being that, when $\sim$ is not in $\mathsf{Rel}(\g,\d)$, the rule  \emph{Ref}$_\sim$ becomes an instance of Case \ref{case3}.\footnote{Otherwise, it is an instance of Case \ref{case1} or of Case \ref{case2}, and then the split-interpolant of the conclusion can be determined as we have shown for \emph{Ref}$_=$, except for the use of the existential quantifier when we have an instance of Case \ref{case2} only and we must quantify away $s$.} 
We consider in detail the cases of \emph{Trans}$_\sim$ and \emph{Sym}$_\sim$.

 Regarding \emph{Trans}$_\sim$, there are four possible partitions of the conclusion: 
 
 \begin{itemize}
  \item $s \sim t , t \sim u , \Gamma_1 \prt \Gamma_2 \Rightarrow \Delta_1 \prt \Delta_2$
  \item $\Gamma_1 \prt s \sim t , t \sim u , \Gamma_2 \Rightarrow \Delta_1 \prt \Delta_2$
  \item $s \sim t , \Gamma_1 \prt t \sim u , \Gamma_2 \Rightarrow \Delta_1 \prt \Delta_2$
  \item $t \sim u , \Gamma_1 \prt s \sim t , \Gamma_2 \Rightarrow \Delta_1 \prt \Delta_2$
  \end{itemize}

For the first two partitions, we find the split-interpolant by reasoning as in Case \ref{case1} with $\Pi_2=\varnothing$ and Case \ref{case2} with $\Pi_1=\varnothing$, respectively. Hence, a split-interpolant for the first and second partitions is:\vspace{0.3cm} 
 
\noindent \scalebox{0.95000}{$$
\framebox{
\infer{s \sim t , t \sim u , \Gamma_1 \prt \Gamma_2 \xRightarrow{C} \Delta_1 \prt \Delta_2}{s \sim u , s \sim t , t \sim u , \Gamma_1 \prt \Gamma_2 \xRightarrow{C} \Delta_1 \prt \Delta_2}\qquad
\infer{\Gamma_1 \prt s \sim t , t \sim u , \Gamma_2 \xRightarrow{C} \Delta_1 \prt \Delta_2}{\Gamma_1 \prt s \sim u , s \sim t , t \sim u , \Gamma_2 \xRightarrow{C} \Delta_1 \prt \Delta_2}
}
$$}\vspace{0.3cm}

For the last two partitions we can proceed as in Case \ref{case1} or as in Case \ref{case2}. By proceeding as in  Case  \ref{case1} we find the following split-interpolants, assuming, respectively, $u\not\in\mathsf{Ter} (s \sim t , \Gamma_1 , \Delta_1)$ and  $s\not\in\mathsf{Ter} (t \sim u , \Gamma_1 , \Delta_1)$:

\begin{footnotesize}
 $$
\framebox{\infer{s \sim t , \g_1 \prt t \sim u, \g_2 \xRightarrow{\forall z (t \sim z \to C[{}_{u}^{z}])} \d_1 \prt \d_2}{s \sim u, s \sim t , t \sim u,\g_1 \prt \g_2 \xRightarrow{C} \d_1 \prt \d_2}
\quad
\infer{t \sim u , \g_1 \prt s \sim t, \g_2 \xRightarrow{\forall z (z \sim t \to C[{}_{s}^{z}])} \d_1 \prt \d_2}{s \sim u, s \sim t , t \sim u ,\g_1 \prt  \g_2 \xRightarrow{C} \d_1 \prt \d_2}
}
$$
\end{footnotesize}

\noindent If, instead, $u\in\mathsf{Ter} (s \sim t , \Gamma_1 , \Delta_1)$ or  $s\in\mathsf{Ter} (t \sim u , \Gamma_1 , \Delta_1)$, then we do not quantify them away and we have:
\begin{footnotesize}
 $$
\framebox{\infer{s \sim t , \g_1 \prt t \sim u, \g_2 \xRightarrow{t \sim u \to C} \d_1 \prt \d_2}{s \sim u, s \sim t , t \sim u,\g_1 \prt \g_2 \xRightarrow{C} \d_1 \prt \d_2}
\quad
\infer{t \sim u , \g_1 \prt s \sim t, \g_2 \xRightarrow{s \sim t \to C} \d_1 \prt \d_2}{s \sim u, s \sim t , t \sim u ,\g_1 \prt  \g_2 \xRightarrow{C} \d_1 \prt \d_2}
}
$$
\end{footnotesize}
 Regarding \emph{Sym}$_\sim$, there are two possible partitions of the conclusion: 
 
 \begin{itemize}
  \item $s \sim t ,  \Gamma_1 \prt \Gamma_2 \Rightarrow \Delta_1 \prt \Delta_2$
  \item $\Gamma_1 \prt s \sim t ,  \Gamma_2 \Rightarrow \Delta_1 \prt \Delta_2$
\end{itemize}
We find the split-interpolant by reasoning as in Case \ref{case1} with $\Pi_2=\varnothing$ and Case \ref{case2} with $\Pi_1=\varnothing$, respectively. Hence we have:

\begin{small}$$
\framebox{
\infer{s \sim t ,  \Gamma_1 \prt \Gamma_2 \xRightarrow{C} \Delta_1 \prt \Delta_2}{t \sim s , s \sim t ,  \Gamma_1 \prt \Gamma_2 \xRightarrow{C} \Delta_1 \prt \Delta_2}\qquad
\infer{\Gamma_1 \prt s \sim t ,  \Gamma_2 \xRightarrow{C} \Delta_1 \prt \Delta_2}{\Gamma_1 \prt t \sim s , s \sim t ,  \Gamma_2 \xRightarrow{C} \Delta_1 \prt \Delta_2}
}
$$\end{small}\end{proof}
\subsection{Partial and linear orders}\label{PO}

Now we consider some well-known order theories. We start with partial orders. In sequent calculus, the theory of partial orders is obtained by extending $\mathsf{Gc^=}$ with the following rules corresponding to the axioms of reflexivity, transitivity and anti-symmetry of a binary relation $\leqslant$. Thus, let $\mathsf{PO} = \mathsf{Gc^=} + \{\textnormal{\emph{Ref}}_\leqslant \, , \, \textnormal{\emph{Trans}}_\leqslant \, , \, \textnormal{\emph{Anti-sym}}_\leqslant\}$:

\begin{center}
 \begin{tabular}{cc}
$
\infer[\infrule{Ref_\leqslant}]{\Gamma \Rightarrow \Delta}{s \leqslant s , \Gamma \Rightarrow \Delta}
$
&
$
 \infer[\infrule{Trans_\leqslant}]{s \leqslant t , t \leqslant u , \Gamma \xRightarrow{} \Delta}{s \leqslant u , s \leqslant t , t \leqslant u , \Gamma \xRightarrow{} \Delta}
$
\\\\
\multicolumn{2}{c}{
$\infer[\infrule{Anti\textnormal{-}sym_\leqslant}]{s \leqslant t , t \leqslant s , \Gamma \xRightarrow{} \Delta}{s = t , s \leqslant t , t \leqslant s , \Gamma \xRightarrow{} \Delta}$
}
 \end{tabular}
\end{center}
Linear orders are obtained by assuming that the partial order $\leqslant$ is also linear, i.e $\mathsf{LO} = \mathsf{PO} + \{ \textnormal{\emph{Lin}}_\leqslant\}$.

$$
\infer[\infrule{Lin_\leqslant}]{\Gamma \Rightarrow \Delta}{s \leqslant t , \Gamma \Rightarrow \Delta & t \leqslant s , \Gamma \Rightarrow \Delta}
$$
Both $\mathsf{PO}$ and $\mathsf{LO}$ are singular geometric theories, hence: 

\begin{cor}
 $\mathsf{LO}$ (hence, $\mathsf{PO}$) has the interpolation property. 
\end{cor}

\begin{proof}
 The procedure for building the interpolants for \emph{Ref}$_\leqslant$ and  \emph{Trans}$_\leqslant$ are the same as those for \emph{Ref}$_\sim$ and \emph{Trans}$_\sim$, respectively, in $\mathsf{EQ}$;  that for \emph{Anti-sym}$_\leqslant$ is like that for \emph{Trans}$_\sim$, save that here there is no need to quantify away any term occurring in the split-interpolant. 
 
 For \emph{Lin}$_\leqslant$,  only one partition of the conclusion has to be considered, namely $\g_1 \prt \g_2 \Rightarrow \d_1 \prt \d_2$. Its interpolant can be found as in Case \ref{case3} of Lemma \ref{th:maehara_singular} with $\Pi_2=\varnothing$, provided that $\leqslant$ is not in $\mathsf{Rel}(\g,\d)$.\footnote{Else, we proceed as in Case \ref{case1} or \ref{case2} and, as for rule \emph{Ref}$\sim$,  in the latter case, when we have to quantify away $s$ and $t$ we do it via existential quantifiers.} Assuming that both $s$ and $t$ are in $\mathsf{Ter}(C1,C_2)$ but not in $\mathsf{Ter}(\g_2,\d_2)$:

$$
\framebox{
  \infer
  {
    \Gamma_1 \prt  \Gamma_2 \xRightarrow{\forall z_1\forall z_2((C_1\lor C_2)[{}_{s}^{z_1}{}_{t}^{z_2}])} \Delta_1 \prt \Delta_2
  }
  {
  s\leqslant t, \Gamma_1 \prt  \Gamma_2 \xRightarrow{C_1} \Delta_1 \prt \Delta_2&t\leqslant s, \Gamma_1 \prt  \Gamma_2 \xRightarrow{C_2} \Delta_1 \prt \Delta_2
  }
  }
  $$
  If, instead,  $s$ or $t$ is in $\mathsf{Ter}(\g_2,\d_2)$, or if it is not in $\mathsf{Ter}(C_1,C_2)$, then it is not  quantified away.
\end{proof}

Unlike $\mathsf{G^=}$ and $\mathsf{EQ}$, the underlying logical calculus of both $\mathsf{PO}$ and $\mathsf{LO}$ is  the classical one. The reason is that linearity is intuitionistically contentious and normally it requires a different, more constructively acceptable, axiomatization that will be considered in Section \ref{sectPPO}.

\subsection{Strict partial and linear orders} 

The theory of strict partial orders consists of the axioms of first-order logic with identity plus the irreflexivity and transitivity of $<$. As we did for $\mathsf{PO}$ and $\mathsf{LO}$, we consider this theory to be based on classical logic, i.e. by adding on top of $\mathsf{Gc}^=$ the following rules: 

\begin{center}
 \begin{tabular}{cc}
 $\infer[\infrule{Irref_<}]{s < s , \Gamma \xRightarrow{} \Delta}{
 }
 \qquad
 \infer[\infrule{Trans_<}]{s < t , t < u , \Gamma \xRightarrow{} \Delta}{s < u , s < t , t < u , \Gamma \xRightarrow{} \Delta}
 $
 
\end{tabular}
\end{center}

Let $\mathsf{SPO}$ be $\mathsf{Gc}^= + \{\textnormal{\emph{Irref}}_< , \textnormal{\emph{Trans}}_<\}$. Total strict partial orders are then obtained assuming that $<$ is also trichotomic, i.e. $\mathsf{SLO} = \mathsf{SPO} + \{ \textnormal{\emph{Trich}}_<\}$:

$$
 \infer[\infrule{Trich_<}]{\Gamma \xRightarrow{} \Delta}{s = t , \Gamma \xRightarrow{} \Delta & s < t , \Gamma \xRightarrow{} \Delta & t < s , \Gamma \xRightarrow{} \Delta}
$$

\begin{cor}
$\mathsf{SLO}$ (hence, $\mathsf{SPO}$) has the interpolation property.

\end{cor}

\begin{proof}
We show how to find the interpolants for \emph{Irref}$_<$ and \emph{Trich}$_<$, while \emph{Trans}$_<$ is identical to \emph{Trans}$_\sim$. 
We start with \emph{Irref}$_<$. There are two possible partitions of its conclusion, namely

\begin{itemize}
 \item $s < s , \Gamma_1 \prt \Gamma_2 \Rightarrow \Delta_1 \prt \Delta_2$
\item $\Gamma_1 \prt s < s , \Gamma_2 \Rightarrow \Delta_1 \prt \Delta_2$
 \end{itemize}

As in Case \ref{case1} with $\Pi_2=\varnothing$  (and $m=0$) and as in Case \ref{case2} with $\Pi_1=\varnothing$  (and $m=0$) of Lemma \ref{th:maehara_singular}, we find the split-interpolant for each partition as follows:  

\

$$
 \framebox{
 \infer{s < s , \Gamma_1 \prt \Gamma_2 \xRightarrow{\bot} \Delta_1 \prt \Delta_2}{
 }
 \qquad
\infer{\Gamma_1 \prt s < s , \Gamma_2 \xRightarrow{\top} \Delta_1 \prt \Delta_2}{\phantom{A^7}
}
 }
 $$

 \

Regarding \emph{Trich}$_<$, we need to consider only one partition of the conclusion, namely $\g_1 \prt \g_2 \Rightarrow \d_1 \prt \d_2$, whose interpolant can be found as  in Case \ref{case3} of Lemma \ref{th:maehara_singular} with $\Pi_2=\varnothing$ when $<$ is not in $\mathsf{Rel}(\g,\d)$.\footnote{Else, as for rule \emph{Ref}$_\sim$, we proceed as in one of Cases \ref{case1} and \ref{case2}.} Assuming that both $s$ and $t$ are in $\mathsf{Ter}(C_1,C_2,C_3)$ but not in $\mathsf{Ter}(\g_2,\d_2)$:

$$
\scalebox{0.90000}{\framebox{
  \infer
  {
    \Gamma_1 \prt  \Gamma_2 \xRightarrow{\forall z_1\forall z_2((C_1\lor C_2\lor C_3)[{}_{s}^{z_1}{}_{t}^{z_2}])} \Delta_1 \prt \Delta_2
  }
  {
 s=t,\g_1\prt\g_2\xRightarrow{C_1}\d_1\prt\d_2& s< t, \Gamma_1 \prt  \Gamma_2 \xRightarrow{C_2} \Delta_1 \prt \Delta_2&t< s, \Gamma_1 \prt  \Gamma_2 \xRightarrow{C_3} \Delta_1 \prt \Delta_2
  }
  }
  }
  $$
  If $s$ or $t$ is in $\mathsf{Ter}(\g_2,\d_2)$, or if it is not in $\mathsf{Ter}(C_1,C_2,C_3)$, then it is not  quantified away.

\end{proof}

\subsection{Apartness}\label{apart} 

We noticed earlier that in intuitionistic theories the identity relation is not always treated as in classical logic. In particular, identity is defined in terms of the more constructively acceptable relation of apartness. Apartness was originally introduced by Brouwer (and later axiomatized by Heyting in \cite{Heyting1956}) to express inequality between real numbers in the constructive analysis of the continuum: whereas saying that two real numbers $a$ and $b$ are unequal only means that the assumption $a = b$ is contradictory, to say that $a$ and $b$ are apart expresses the constructively stronger requirement that their distance on the real line can be effectively measured, i.e. that $|\, a - b \,| > 0$ has a constructive proof.  Classically, inequality and apartness coincide, but intuitionistically two real numbers can be unequal without being apart. The theory of apartness consists of intuitionistic  first-order logic plus the irreflexivity and splitting of $\neq$. Following \cite{Negri1999}, the theory of apartness is formulated by adding on top of $\mathsf{Gi}$ the following rules:\footnote{Notice that Negri's underlying calculus is a quantifier-free version of $\mathsf{Gi}$.} 

\begin{center}
 \begin{tabular}{cc}
 $\infer[\infrule{Irref_{\neq}}]{s \neq s , \Gamma \xRightarrow{} A}{
 }
 \qquad
 \infer[\infrule{Split_{\neq}}]{s \neq t ,\Gamma \xRightarrow{} A}{s \neq u , s \neq t, \Gamma \xRightarrow{} A&t \neq u , s \neq t, \Gamma \xRightarrow{} A}
 $
 
\end{tabular}
\end{center}

Let $\mathsf{AP} = \mathsf{Gi} + \{ \textnormal{\emph{Irref}}_{\neq} \, , \, \textnormal{\emph{Split}}_{\neq}\}$. Given that these two rules are singular geometric rules, it follows that:

\begin{cor}\label{ap}
$\mathsf{AP}$ has the interpolation property. 
\end{cor}

\begin{proof}
As above, we show how to find the interpolants for \emph{Irref}$_{\neq}$ and \emph{Split}$_{\neq}$. The former is identical to that of \emph{Irref}$_<$ in $\mathsf{SPO}$.
 
 In the case of \emph{Split}$_{\neq}$, there are two possible partitions of the conclusion: 
 
 \begin{itemize}
  \item $s \neq  t , \Gamma_1 \prt \Gamma_2 \Rightarrow  \prt A$
  \item $\Gamma_1 \prt s \neq t  ,  \Gamma_2 \Rightarrow  \prt A$

  \end{itemize}

For the first partition, we use  Case \ref{case1}  of Lemma \ref{th:maehara_singular} with $\Pi_2=\varnothing$. Thus, if $u\not\in\mathsf{Ter}(s\neq t,\g_1)$ and $u\in \mathsf{Ter}(C_1,C_2)$
, a split-interpolant for the first partition is:

$$
\framebox{
\infer{s\neq t, \Gamma_1 \prt \Gamma_2 \xRightarrow{\forall z(C_1[{}_{u}^{z}]\lor C_2[{}_{u}^{z}])} \prt A}{s \neq  u , s \neq t ,\Gamma_1 \prt \Gamma_2 \xRightarrow{C_1} \prt A\quad &t \neq  u , s \neq t, \Gamma_1 \prt \Gamma_2 \xRightarrow{C_2} \prt A }
}
$$

For the second partition, we use  Case \ref{case2} of Lemma \ref{th:maehara_singular} with $\Pi_1=\varnothing$. Thus, if $u\not\in\mathsf{Ter}(s\neq t,\g_2,A)$
, a split-interpolant for the second partition is:

$$
\framebox{
\infer{ \Gamma_1 \prt s\neq t,\Gamma_2 \xRightarrow{\exists z(C_1[{}_{u}^{z}]\wedge C_2[{}_{u}^{z}])} \prt A}{\Gamma_1 \prt s \neq  u , s \neq t , \Gamma_2 \xRightarrow{C_1} \prt A\quad & \Gamma_1 \prt t \neq  u , s \neq t,\Gamma_2 \xRightarrow{C_2} \prt A }
}
$$
When $u$ is, respectively, in $\mathsf{Ter}(s\neq t,\g_1)$ or in $\mathsf{Ter}(s\neq t,\g_2,A)$, as well as when it is not in $\mathsf{Ter}(C_1,C_2)$, we do not quantify it away .\end{proof}

\subsection{Positive partial and linear orders}\label{sectPPO} 

Just like apartness is a positive version of inequality, so excess $\nleqslant$ is a positive version of the negation of a partial order $\leqslant$. Excess relation was introduced by von Plato in \cite{vonPlato2001} and has been further investigated by Negri in \cite{Negri1999}.   
The theory of positive partial orders consists of intuitionistic first-order logic plus the irreflexivity and co-transitivity of $\nleqslant$.\footnote{Co-transitivity and splitting should not be confused. In particular, splitting (along with irreflexivity) gives symmetry, whereas co-transitivity does not. This is what distinguishes apartness (which is symmetric) from excess (which in general is not).} Let $\mathsf{PPO} = \mathsf{Gi} + \{ \textnormal{\emph{Irref}}_{\nleqslant} \, , \, \textnormal{\emph{Co-trans}}_{\nleqslant}\}$

\begin{center}
 \begin{tabular}{cc}
 $\infer[\infrule{Irref_{\nleqslant}}]{s \nleqslant s , \Gamma \xRightarrow{} A}{
 }
 \qquad
 \infer[\infrule{Co\textnormal{-}trans_{\nleqslant}}]{s \nleqslant t ,\Gamma \xRightarrow{} A}{s \nleqslant u , s \nleqslant t, \Gamma \xRightarrow{} A&u \nleqslant t , s \nleqslant t, \Gamma \xRightarrow{} A}
 $
 
\end{tabular}
\end{center}

The theory of positive linear orders extends the theory of positive partial orders with the asymmetry of $\nleqslant$. Specifically, let $\mathsf{PLO} = \mathsf{PPO} + \{\textnormal{\emph{Asym}}_\nleqslant\}$:

$$
\infer[\infrule Asym_{\nleqslant}]{s \nleqslant t,t \nleqslant s,\Gamma\To A}{}
$$
 Given that all these rules are singular geometric, from Theorem \ref{CraigSingular} it follows that

\begin{cor}
$\mathsf{PPO}$ and in $\mathsf{PLO}$ have the interpolation property. 
\end{cor}

\begin{proof}
The cases of \emph{Irref$\,{}_{\nleqslant}$} and of \emph{Co-Trans$\,{}_{\nleqslant}$} are like the  analogous cases for rules \emph{Irref$\,{}_{\neq}$} and  \emph{Split$\,{}_{\neq}$} and the split-interpolants can be obtained by those in the proof of Corollary \ref{ap}. For rule \emph{Asym$\,{}_{\nleqslant}$} we have four possible partitions  of the conclusion
 \begin{itemize}
  \item $s \nleqslant t , t \nleqslant s , \Gamma_1 \prt \Gamma_2 \Rightarrow \prt A$
  \item $\Gamma_1 \prt s \nleqslant t , t \nleqslant  s , \Gamma_2 \Rightarrow  \prt A$
  \item $s \nleqslant t , \Gamma_1 \prt t \nleqslant s , \Gamma_2 \Rightarrow  \prt A$
  \item $t \nleqslant s , \Gamma_1 \prt s \nleqslant t , \Gamma_2 \Rightarrow  \prt A$
  \end{itemize}
  
  Their split-interpolants are like those for rule \emph{Anti-sym}$_\leqslant$, except that here we have a 0-premise rule.  For the first and second partitions we have, respectively:

$$
\framebox{
\infer{s \nleqslant t , t \nleqslant s , \Gamma_1 \prt \Gamma_2 \xRightarrow{\bot}  \prt A}{\phantom{C}}\qquad\qquad
\infer{\Gamma_1 \prt s \nleqslant t , t \nleqslant s , \Gamma_2 \xRightarrow{\top}  \prt A}{\phantom{C}}
}
$$

\noindent Finally,  for the last two partitions we have, respectively:

 $$
\framebox{\infer{s \nleqslant t , \g_1 \prt t \nleqslant s, \g_2 \xRightarrow{t \nleqslant s \to \bot} \prt A}{\phantom{C}}
\quad
\infer{t \nleqslant s , \g_1 \prt s \nleqslant t, \g_2 \xRightarrow{s \nleqslant t\to\bot}  \prt A}{\phantom{C}}
}
$$
\end{proof}
 
 To conclude, we have shown (Lemma \ref{th:maehara_singular}) how to extend Maehara's lemma to extensions of classical and intuitionistic sequent calculi with singular geometric rules and provided a number of interesting examples of singular geometric rules that are important both in logic and mathematics, especially in order theories. In particular, we have shown that Lemma \ref{th:maehara_singular} covers first-order logic with identity and its extension with the theory of (strict) partial and linear orders. We have also proved that the same holds for the intuitionistic theories of apartness, as well as for positive partial and linear order. Along the way, we have also provided a cut-elimination theorem for geometric extensions $\mathsf{Gi^g}$ of the intuitionistic single-succedent calculus $\mathsf{Gi}$.

\vspace{0,5cm}

\textbf{Acknowledgements:} We are very grateful to Birgit Elbl for precious comments and helpful discussions on various points. We also thank an anonymous referee for valuable suggestions that have helped to generalize our main result as well as to improve its exposition.

  \bibliographystyle{plain}
  \bibliography{RefList}  

\providecommand{\noopsort}[1]{}
\begin{thebibliography}{10}

\bibitem{BaazIemhoff2005}
M.~Baaz and R.~Iemhoff.
\newblock On interpolation in existence logics.
\newblock In {\em Logic for Programming, Artificial Intelligence, and
  Reasoning}, volume 3835 of {\em Lecture Notes in Computer Science}, pages
  697--711. Springer, 2005.

\bibitem{BonacinaJohansson2015}
M.~Bonacina and M.~Johansson.
\newblock Interpolation systems for ground proofs in automated deduction: a
  survey.
\newblock {\em Journal of Automated Reasoning}, 54:353--390, 2015.

\bibitem{Casari2006}
E.~Casari.
\newblock {\em La matematica della verit\`{a}}.
\newblock Bollati Boringhieri, 2006.

\bibitem{Craig1957}
W.~Craig.
\newblock Three uses of the {H}erbrand-{G}entzen theorem in relating model
  theory and proof theory.
\newblock {\em The Journal of Symbolic Logic}, 22(3):269--285, 1957.

\bibitem{Dragalin1988}
A.G. Dragalin.
\newblock {\em Mathematical Intuitionism: Introduction to Proof Theory}.
\newblock American Mathematical Society, 1988.

\bibitem{FittingKuznets2015}
M.~Fitting and R.~Kuznets.
\newblock Modal interpolation via nested sequents.
\newblock {\em Annals of Pure and Applied Logic}, 166(3):274--305, 2015.

\bibitem{GabbayMaksimova2005}
D.~Gabbay and L.~Maksimova.
\newblock {\em Interpolation and Definability: Modal and Intuitionistic
  Logics}.
\newblock Oxford University Press, 2005.

\bibitem{Gallier2015}
J.~Gallier.
\newblock {\em Logic for Computer Science}.
\newblock Dover, 2nd edition, 2015.

\bibitem{Gentzen1969a}
G.~Gentzen.
\newblock Investigation into logical deductions.
\newblock In M.~E. Sazbo, editor, {\em The collected papers of Gerhard
  Gentzen}, chapter~3, pages 68--131. North-Holland, 1969.

\bibitem{Heyting1956}
A.~Heyting.
\newblock {\em Intuitionism. An introduction}.
\newblock North-Holland, 1956.

\bibitem{Kuznets2018}
R.~Kuznets.
\newblock Multicomponent proof-theoretic method for proving interpolation
  property.
\newblock {\em Annals of Pure and Applied Logic}, 169(2):1369--1418, 2018.

\bibitem{Maehara1960}
S.~Maehara.
\newblock On the interpolation theorem of {Craig}.
\newblock {\em Suugaku}, 12:235--237, 1960.
\newblock (in Japanese).

\bibitem{Negri1999}
S.~Negri.
\newblock Sequent calculus proof theory of intuitionistic apartness and order
  relations.
\newblock {\em Archive for Mathematical Logic}, 38(8):521--547, 1999.

\bibitem{Negri2003}
S.~Negri.
\newblock Contraction-free sequent calculi for geometric theories with an
  application to {B}arr's theorem.
\newblock {\em Archive for Mathematical Logic}, 42(4):389--401, 2003.

\bibitem{NegrivonPlato1998}
S.~Negri and J.~von Plato.
\newblock Cut elimination in the presence of axioms.
\newblock {\em The Bulletin of Symbolic Logic}, 4(4):418--435, 1998.

\bibitem{NegrivonPlato2001}
S.~Negri and J.~von Plato.
\newblock {\em Structural Proof Theory}.
\newblock Cambridge University Press, 2001.

\bibitem{NegrivonPlato2011}
S.~Negri and J.~von Plato.
\newblock {\em Proof Analysis: A Contribution to Hilbert's Last Problem}.
\newblock Cambridge University Press, 2011.

\bibitem{vonPlato2001}
{\noopsort{Plato}}{J. von Plato}.
\newblock Positive lattices.
\newblock In P.~Schuster, U.~Berger, and H.~Osswald, editors, {\em Reuniting
  the Antipodes - Constructive and Nonstandard Views of the Continuum}, volume
  306 of {\em Synthese Library}, pages 185--197. Kluwer, 2001.

\bibitem{RasgaCarnielliSernadas2009}
J.~Rasga, W.~Carnielli, and C.~Sernadas.
\newblock Interpolation via translations.
\newblock {\em Mathematical Logic Quarterly}, 55(5):515--534, 2009.

\bibitem{Takeuti1987}
G.~Takeuti.
\newblock {\em Proof Theory}, volume~81 of {\em Studies in Logic and the
  Foundations of Mathematics}.
\newblock North-Holland, 2nd edition, 1987.

\bibitem{TroelstraSchwichtenberg2000}
A.S. Troelstra and H.~Schwichtenberg.
\newblock {\em Basic Proof Theory}.
\newblock Cambridge University Press, 2nd edition, 2000.

\end{thebibliography}
\end{document}